\documentclass[11pt]{article}
\usepackage{amsmath,amssymb,amsfonts,amsthm}
\usepackage{mathrsfs}
\usepackage{indentfirst}
\usepackage[pdftex]{color,graphicx}
\usepackage{subfig}
\usepackage[pdftex,bookmarks,unicode,colorlinks]{hyperref}
\usepackage{enumerate} 
\usepackage[numbers]{natbib} 
\usepackage{yfonts} 

\usepackage{caption} 
\theoremstyle{plain}
\newtheorem{theorem}{Theorem}[section]
\newtheorem{proposition}[theorem]{Proposition}
\newtheorem{corollary}[theorem]{Corollary}
\newtheorem{lemma}[theorem]{Lemma}

\theoremstyle{remark}
\newtheorem{remark}[theorem]{Remark}
\newtheorem{example}[theorem]{Example}

\theoremstyle{definition}

\newcommand{\E}{\mathbf E}
\renewcommand{\P}{\mathbf P}

\newcommand{\R}{\mathbb R}
\newcommand{\N}{\mathbb N}

\renewcommand{\S}{\mathbb S}

\newcommand{\Ro}[1]{\mathbb R^{#1}\setminus \{0\}}
\renewcommand{\C}{\mathcal C}

\newcommand{\A}{\mathcal A}

\newcommand{\e}{\epsilon}
\newcommand{\F}{\mathcal F}
\newcommand{\B}{\mathcal B}
\newcommand{\X}{\mathcal X}

\newcommand{\ind}{\mathbf 1}

\newcommand{\Ad}{\mathrm{Ad}}
\newcommand{\I}{\mathbf I}
\newcommand{\Lip}{\mathrm{Lip}}
\renewcommand{\H}{\mathcal H}

\numberwithin{equation}{section} 

\begin{document}

\title{\bf\Large Rolling with Random Slipping and Twisting: \\A Large Deviation Point of View}
\author{\bf\normalsize{
Qiao Huang$^{1,}$\footnote{Email: \texttt{hq932309@alumni.hust.edu.cn}},
Wei Wei$^{1,}$\footnote{Email: \texttt{weiw16@hust.edu.cn}},
Jinqiao Duan$^{1,2,}$\footnote{Email: \texttt{duan@iit.edu}}
} \\[10pt]
\footnotesize{$^1$Center for Mathematical Sciences, Huazhong University of Science and Technology,} \\
\footnotesize{Wuhan, Hubei 430074, P.R. China.} \\[5pt]
\footnotesize{$^2$Department of Applied Mathematics, Illinois Institute of Technology, Chicago, IL 60616, USA.}
}

\date{}
\maketitle
\vspace{-0.3in}

\begin{abstract}
  We study a rolling model from the perspective of probability. More precisely, we consider a Riemannian manifold rolling against Euclidean space, where the rolling is coupled with random slipping and twisting. The system is modelled by a stochastic differential equation of Stratonovich-type driven by semimartingales, on the orthonormal frame bundle. The stability of the system is examined via large deviations. 
  We prove the large deviation principles for the projection curves on the base manifold and their horizontal lifts respectively, provided that the large deviation holds for the random Euclidean curves as semimartingales. The large deviation results for the case of compact manifolds and two special cases of noncompact manifolds are established.
\end{abstract}

\section{Introduction}

Differential geometry has been inseparably related to classical mechanics, due to its broader applications to various mechanical systems. A well-known system 
is the sphere \emph{rolling} on the plane \emph{without} slipping or twisting, which is an early interesting study
in non-holonomic mechanics \cite{Cha97,Cha03}. Nowadays the general systems of rolling a manifold are often studied with contributions from intrinsic Riemannian geometry, sub-Riemannian geometry \cite{Mon02,Sha97} and geometric control theory \cite{AS04}. 
A geometric operation, which is known as \emph{Cartan's development} \cite{Car25,Nom78}, plays an essential role in the rolling model \cite{MG14}.
The concept of development plays a fundamental role in defining Brownian motion on a manifold. This leads to an outbreak of stochastic analysis on Riemannian manifolds in the past decades \cite{Hsu02,Str00}.
The idea of constructing Brownian trajectories on manifold is similar to the procedure of rolling sphere. Intuitively, one can draw a Brownian path $B_t$ in $\R^d$, and then one can consider the system of manifold $M$ rolling against $\R^d$ following the path $B_t$ (see Fig.~\ref{rolling-bm} for a visualization). Although the precise definition uses a less common version of Cartan's development and parallel transport, this simple notion allows one to recover the Laplace-Beltrami operator $\Delta_M$ of the manifold. It is usually interpreted that Brownian paths are the ``integral curves'' for $\Delta_M$. This unprecise assertion vividly introduces the idea that second order differential operators induce ``diffusions'' on the manifold. This point of view has been examined maturely in the study of stochastic differential equations (SDEs) on manifolds \cite{Bis84,IW89}.

\begin{figure}[!bth]
  \centering
  \includegraphics[width=0.8\textwidth]{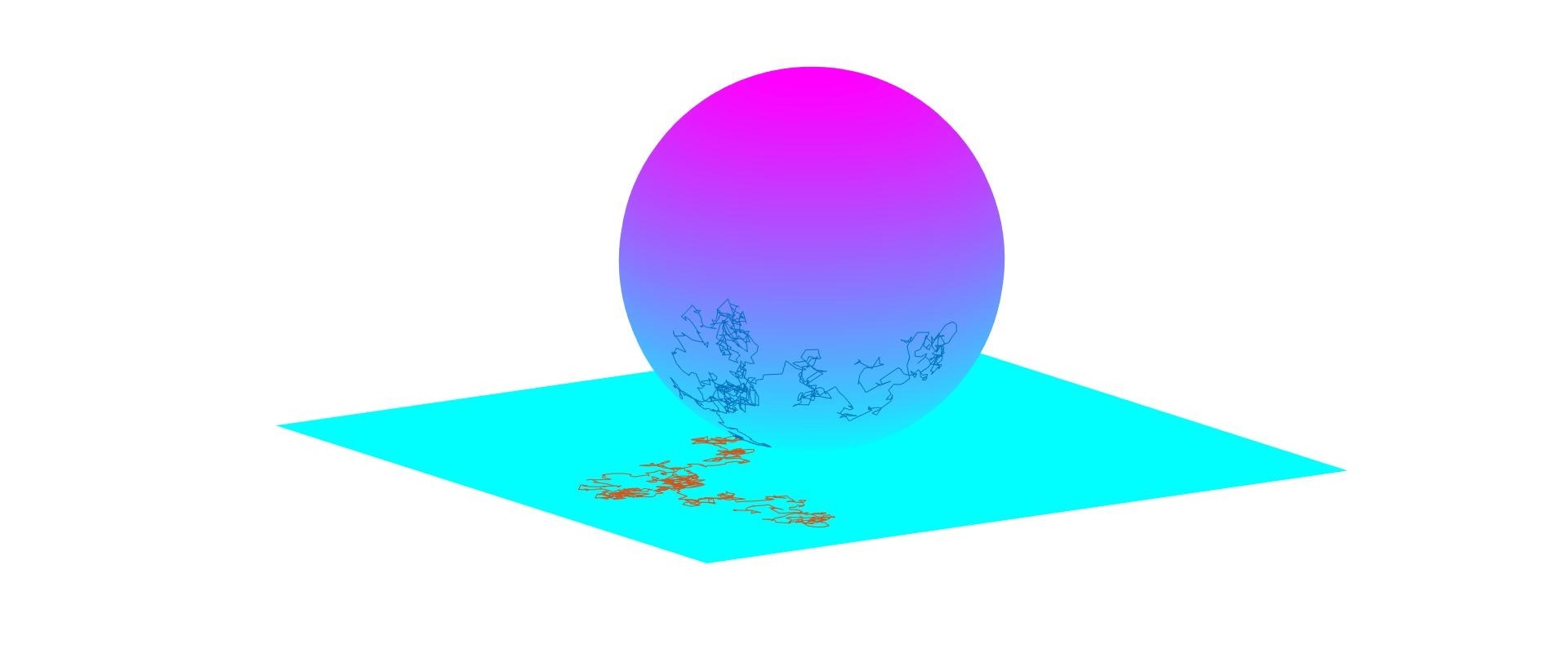}
  \caption{Brownian motion on the sphere $\S^2$ by rolling along a Brownian motion on $\R^2$.} \label{rolling-bm}
\end{figure}

Our aim in the present paper is to model the system of rolling \emph{with} random slipping and twisting in higher dimensions, and then study the stability of the system. We denote by $H_\xi$ the standard horizontal vector field on the orthonormal frame bundle $OM$ corresponding to $\xi\in\R^d$. Given a smooth curve $\gamma = \{\gamma_t\}_{0\le t\le T}$ in Euclidean space $\R^d$, the system of rolling along $\gamma$ without slipping or twisting is modelled by Cartan's development as the following ordinary differential equation (ODE) on $OM$,
\begin{equation}\label{development}
  \dot u_t = H_{\dot\gamma_t}(u_t).
\end{equation}
Projecting the solution curve $u = \{u_t\}_{0\le t\le T}$ onto the base manifold $M$, we get a curve $x = \{x_t\}_{0\le t\le T}$, which is exactly the trace left on the manifold $M$ by rolling $M$ along the pathway $\gamma$ on $\R^d$. If $\gamma$ is a straight line on the Euclidean space, then the projection curve $x$ is a geodesic on $M$ starting from $\pi(u_0)$ with initial velocity $u_0(\dot\gamma_0)$. Let $D = d(d-1)/2$ and let $\{A_1,\cdots, A_D\}$ be an orthonormal basis of $\mathfrak{so}(d)$, the space of skew-symmetric matrices in dimension $d$. For $A\in \mathfrak{so}(d)$, we denote by $A^*$ the fundamental vertical vector field on $OM$ induced by $A$. Let $\{e_1,\cdots,e_d\}$ be the canonical basis of $\R^d$. Denote for shorthand $H_i:= H_{e_i}$. If we add the slipping and twisting ingredients (in a random fashion) into the rolling model, the differential equation on the bundle $OM$ then becomes
\begin{equation}\label{rolling-st}
  d u_t = H_i(u_t)\circ d \tilde\gamma^i_t + A_\alpha^*(u_t)\circ d W_t^\alpha,
\end{equation}
where $\tilde\gamma$ is an equivalent curve involving the slipping constituent, which is regarded as a randomization of the original Euclidean curve $\gamma$, $W = \{W_t\}_{t\ge0}$ is a $D$-dimensional standard Brownian motion, the symbol $\circ$ means that the stochastic differential is in Stratonovich sense. Hereinafter, we use \emph{Einstein's convention} that the repeated indices in a product will be summed automatically. The random twisting is indicated in the second term, i.e., the vertical component, of right hand side of the last equation. Examples for the rolling models with specific patterns of slipping or twisting are given in Section \ref{sec-exmp}.

Recently, there are several works concerned with similar models as \eqref{development} or \eqref{rolling-st}. This includes \cite{Li16a}, where the homogenization of a kind of random perturbed geodesic equation was studied. The author took $\tilde\gamma_t = e_1 t$ in \eqref{rolling-st} and add one more deterministic vertical perturbation on the right hand side, the horizontal and deterministic vertical components are in normal scale while the random vertical components are in `fast' scale. In \cite{ABT15}, the authors introduced a diffusion process with a finite speed of propagation which they call kinetic Brownian motion. This diffusion process is exactly modelled by \eqref{rolling-st} with $\tilde\gamma_t = e_1 t$ as well, but with the family of Lie elements $\{A_\alpha\}_{\alpha=1}^D$ replaced by $\{\tilde A_i\}_{i=2}^d$, $\tilde A_i = e_i \otimes e_1^* - e_1 \otimes e_i^*$.
That is, the kinetic Brownian motion is the trace curve of rolling along the straight line pointing to the first Euclidean direction with unit speed, coupled with certain random twisting. Both papers provide a sort of interpolation between geodesic and Brownian motions. The small mass limit of Langevin equations on a Riemannian manifold was studied in \cite{BHV17}, in the presence of damping and external force. Lifting the resulting stochastic differential equation to the orthogonal frame bundle, it becomes \eqref{development} again, with $\dot\gamma$ solving an SDE on $\R^d$.

We are concerned with a lightly perturbed version of \eqref{rolling-st}, that is, the system of rolling perturbed by small random slipping and twisting, which is modelled by the following stochastic differential equation on $OM$,
\begin{equation*}
  d u^\e_t = H_{i}(u^\e_t) \circ d\gamma^{\e,i}_t + A_\alpha^*(u^\e_t)\circ d W_t^{\e,\alpha},
\end{equation*}
where $\e>0$ is a small parameter, $W^\e := \sqrt\e W$, the curve $\gamma^\e$ is a randomly perturbed version of $\gamma$ indicating the slipping, which is independent of $W^\e$. When $\e\to 0$, the process $W^\e$ converges in distribution to $0$. If in addition, $\gamma^\e$ converges in distribution to the original curve $\gamma$ as $\e\to0$, then the classical stability theory of stochastic differential equations tells that the bundle-valued curve $u^\e$ converges in distribution to the deterministic curve $u$ in \eqref{development} (see \cite[Section IX.6]{JS13}), and then the continuity of the projection of $OM$ to $M$ yields that the trace curve $x^\e$ converges in distribution as well, to the curve $x$ which is the projection of the curve $u$. This means that the stability in distribution sense holds for the rolling systems, in the terminology of dynamics.

A further question is how to qualify the path-wise stability or instability. The theory of \emph{large deviations} shows its power here. Roughly speaking, the theory of large deviations is concerned with the asymptotic estimation of probabilities of rare events. If the large deviation principle (LDP) holds for the trace family $\{x\}_{\e>0}$, then the probability of rare events that the sample path of trace process $x^\e$ is not close to the limit curve $x$ is exponentially small in $\e$. In this sense, we can say that the rolling system is `exponentially' stable. So our \emph{question} is: if the family of Euclidean curves $\{\gamma^\e\}_{\e>0}$ satisfies a large deviation principle, is it true that the family of bundle-valued curves $\{u^\e\}_{\e>0}$ or manifold-valued curves $\{x^\e\}_{\e>0}$ still satisfies a large deviation principle?

We refer to \cite{DZ98,DE11} for more details on the large deviation theory. We would like to point out that the classical large deviation theorems for Brownian motions and random walks were generalized to the setting of complete Riemannian manifolds \cite{KRV18}.

In general, the perturbed curve $\gamma^\e$ is a continuous \emph{semimartingale} (not necessarily Markovian) for each $\e>0$. Therefore, our goal is to prove that the solution family of the following semimartingale-driven stochastic differential equation satisfies a large deviation principle:
\begin{equation}\label{SDE-exmp}
  dX^\e_t = F(X^\e_{t})dY^\e_t, \quad X^\e_0 = x_0\in\R^d,
\end{equation}
provided that the family of noise $\{Y^\e\}$ satisfies the large deviation principle. A similar problem of large deviations for SDE \eqref{SDE-exmp} was investigated in \cite{Gar08,Gan18}. In both papers, the \emph{exponential tightness} assumption on the family $\{(X^\e,Y^\e)\}$ was proposed to prove the LDP for the solution family $\{X^\e\}$, provided that the LDP holds for $\{Y^\e\}$. By the classical large deviation theory, in the presence of exponential tightness of the family $\{(X^\e,Y^\e)\}$, to prove that this family satisfies the large deviation principle with certain rate function, it is enough to assume it holds and then identify the rate function to make sure that this rate function does not depend on the choice of subsequences. The results in those two papers do not completely solve the preceding problem, because the assumption of exponential tightness of $\{(X^\e,Y^\e)\}$ is rather strong and not easy to verify, and it requires additional conditions on the solutions $X^\e$. We will reconsider the large deviations of semimartingale-driven SDEs in the presence of controls, with all assumptions made only in terms of the driven noise and controls (see Section \ref{sec-LDP}). The result of LDP for semimartingale-driven SDEs with no assumptions on solutions $X^\e$ (Proposition \ref{LDP-control}) are new to our knowledge, and are of independent interest.

The main contributions of this paper are Theorem \ref{cpct}, Theorem \ref{non-cpct-1} and Theorem \ref{non-cpct-2}. The first theorem aims to build the large deviations for the rolling problems in the case of compact manifolds, while the other two deal with the case of noncompact manifolds. When the rolled manifold is compact, we can accomplish the task for each Euclidean curve $\gamma^\e$ being general semimartingale. But in the noncompact case, we only treat for the curve $\gamma^\e$ being some special semimartingales. Theorem \ref{non-cpct-1} requires that each Euclidean curve $\gamma^\e$ is locally of finite variation, and Theorem \ref{non-cpct-2} is devoted for each $\gamma^\e$ satisfying an SDE driven by Brownian noise.

The paper is organized as follows. In the next section, we prove the LDP for a class of controlled stochastic differential equations driven by semimartingales. Some technical proofs are left in Appendix. The large deviations for stochastic differential equations of Stratonovich-type are also given there. In Sections \ref{sec-LDP-rolling}, we establish various large deviation principles for the rolling systems. Subsection \ref{sec-cpct} deals with the rolling on compact manifolds, the random Euclidean curves are general continuous semimartingales. Subsection \ref{sec-noncpct} is devoted to the noncompact case, where Euclidean curves are assumed to be either locally of finite variation or driven by stochastic differential equations. Finally, Section \ref{sec-exmp} is reserved for several typical examples of rolling systems.

\section{Large deviations for controlled stochastic differential equations}\label{sec-LDP}

The aim of this section is to present a large deviation principle for the controlled stochastic differential equation. The reason why the `control' appears is that the horizontal lifts of the projections of solution processes satisfy a type of stochastic differential equations with controls, as we will see in Lemma \ref{transform}.

First of all, we recall some definitions in the large deviation theory. Let $\X$ be a topological space equipped with Borel $\sigma$-algebra $\B(\X)$. A \emph{good rate function} $I$ is a lower semicontinuous mapping $I:\X\to [0,\infty]$ such that for all $\alpha\in [0,\infty)$, the level set $\Phi_I(\alpha):=\{x\in\X: I(x)\le\alpha\}$ is a closed, compact subset of $\X$. A family of probability measures $\{\mu_\e\}_{\e>0}$ on $(\X,\B(\X))$ is said to satisfy the \emph{large deviation principle} (LDP) with a good rate function $I$ if, for all $\Gamma\in\B(\X)$,
\begin{equation*}
  -\inf_{x\in\Gamma^\circ} I(x) \le \liminf_{\e\to0} \e\log \mu_\e(\Gamma) \le \limsup_{\e\to0} \e\log \mu_\e(\Gamma) \le -\inf_{x\in\overline\Gamma} I(x).
\end{equation*}
A family of $\X$-valued random elements $\{X^\e\}_{\e>0}$ is said to satisfies the large deviation principle if the family of probability measures induced by each $X^\e$ on $\X$ satisfies the large deviation principle. We refer to \cite{DZ98} for more details of the large deviation theory.

\subsection{Exponential tightness}

A family of probability measures $\{\mu_\e\}_{\e>0}$ on $(\X,\B(\X))$ is said to be \emph{exponentially tight} if for every $\alpha<\infty$, there exists a compact set $K \subset \X$ such that
\begin{equation*}
  \limsup_{\e\to0} \e \log \mu_\e(K^c) < -\alpha.
\end{equation*}
Or equivalently, for every $0<\delta<1$, there exists a compact set $K\subset \X$ and $\e_0>0$, such that for all $0<\e<\e_0$,
\begin{equation*}
  [\mu_\e(K^c)]^\e < \delta.
\end{equation*}
Similarly, a family of $\X$-valued random elements $\{X^\e\}_{\e>0}$ is said to be exponentially tight if the family of induced probability measures on $\X$ is exponentially tight. If $\X$ has a countable base, from \cite[Lemma 4.1.23]{DZ98} we know that an exponentially tight sequence of measures has a subsequence that satisfies the large deviation principle with some good rate function (see \cite[Theorem (P)]{Puh91}). In case of $\X=\R^d$, the exponential tightness is equivalent to the exponential stochastic boundness found in \cite{Gar08}.

Now denote by $\C^d:=\C(\R_+;\R^d)$ the space of all $\R^d$-valued continuous functions on $\R_+$, equipped with the local uniform topology which is the topology of uniform convergence on compact intervals. It is well known that $\C^d$ is a Polish space. We associate with $\C^d$ the Borel $\sigma$-algebra $\B(\C^d)$. The exponential tightness of a family of probability measures on $(\C^d,\B(\C^d))$ is implied by the LDP with a good rate function (see, e.g., \cite[Lemma 3.5]{FK06}). For $\rho>0$, $T>0$ and $x = \{x(t)\}_{t\ge0}\in \C^d$, define the uniform norm and modulus of continuity
\begin{equation*}
  \|x\|_T := \sup_{0\le t\le T}|x(t)|, \quad
  w_T(x,\rho) := \sup_{\begin{subarray}{c}
                        |t-s|\le\rho \\
                        0\le t,s \le T
                      \end{subarray}
  } |x(t)-x(s)|.
\end{equation*}

The following lemma is a criterion for the exponential tightness of probability measures in $\C^d$, referring to \cite[Theorem 4.2]{Puh91}.
\begin{lemma}\label{exp-tight}
  A family of probability measures $\{\mu_\e\}_{\e>0}$ on $\C^d$ is exponentially tight if and only if the following two statements hold:

  (i). for every $T>0$,
  \begin{equation*}
    \lim_{a\to\infty} \limsup_{\e\to0} \e\log \mu_\e\left( x\in \C^d: \|x\|_T \ge a \right) = -\infty,
  \end{equation*}

  (ii). for every $T>0$, $\eta>0$,
  \begin{equation*}
    \lim_{\rho\to0} \limsup_{\e\to0} \e \log \mu_\e\left( x\in \C^d: w_T(x,\rho) \ge\eta \right) = -\infty.
  \end{equation*}
\end{lemma}

\subsection{Large deviation results}

For each $\e>0$, we have a filtered probability space $(\Omega^\e,\F^\e,\{\F^\e_t\}_{t\ge0},\P^\e)$ endowed with an $n$-dimensional continuous semimartingale $Y^\e$ and an $l$-dimensional continuous adapted process $U^\e$. We also have a locally Lipschitz function $F: \R^d\times\R^l \to \R^d\otimes (\R^n)^*$, $(x,u)\mapsto (F^j_i(x,u); 1\le i\le n, 1\le j\le d)$ with linear growth, so that each controlled stochastic differential equation
\begin{equation}\label{SDE-Ito}
  dX^\e_t = F(X^\e_{t},U^\e_{t}) dY^\e_t, \quad X^\e_0 = x_0\in\R^d,
\end{equation}
has a unique global solution $X^\e$.

We assume that for each $\e>0$, the canonical decomposition of the continuous semimartingale $Y^\e$ is
\begin{equation*}
  Y^\e = M^\e + A^\e,
\end{equation*}
where $M^\e$ is a continuous local martingale, and $A^\e$ is a continuous predictable process with locally finite variation. We associate with $Y^\e$ an increasing process $G(Y^\e)$ via
\begin{equation}\label{G}
  G(Y^\e)_t := |V(A^\e)|_t + \frac{1}{\e}|\langle M^\e, M^\e \rangle|_t,
\end{equation}
where $V(A^\e)$ denotes the variation process associated to $A^\e$, the bracket $\langle M^\e,M^\e\rangle$ denotes the quadratic variation of $M^\e$, which is an $\R^n\otimes\R^n$-valued process.


For functions $f\in \C^{d\times n},y\in \C^n$ with $y$ of finite variation locally, we define
$$(f\cdot y)(t):=\lim_{\|\Delta\|\to 0} \sum_{i} f(t_i) (y(t_{i+1}) - y(t_i)),$$
where $\Delta = \{t_i\}_{i=0}^k$ is a partition of the interval $[0,t]$ with $0=t_0\le t_1\le \cdots \le t_k=t$, and $\|\Delta\| := \max_{1\le i\le k}|t_{i+1}-t_i|$ denotes the mesh of $\Delta$.

We are now prepared to present some large deviation results for the controlled system \eqref{SDE-Ito}. The proofs for the forthcoming Proposition \ref{LDP-control} and Corollary \ref{local-Lip} are quite involved, and will be left into Appendix.
\begin{proposition}\label{LDP-control}
  Let $F$ be bounded and global Lipschitz and $X^\e$ be the solution of \eqref{SDE-Ito} for each $\e>0$. Assume the family $\{G(Y^\e)_t\}_{\e>0}$ defined in \eqref{G} is exponentially tight for each $t>0$, and the family $\{(Y^\e,U^\e)\}_{\e>0}$ satisfies the LDP with a good rate function $I'$. 
  Then

  (i). the family $\{(X^\e,Y^\e,U^\e)\}_{\e>0}$ satisfies the LDP with a good rate function $I$,

  (ii). the rate function $I$ is given by
  \begin{equation}\label{rate-func}
    I(x,y,u) =
    \begin{cases}
      I'(y,u), & x = F(x,u)\cdot y, y \text{ is locally of finite variation}, \\
      \infty, & \text{otherwise}.
    \end{cases}
  \end{equation}
  In particular, the family $\{X^\e\}_{\e>0}$ satisfies the LDP with the following good rate function:
  \begin{equation*}
    I^\flat(x) =\inf\{ I'(y,u): x = F(x,u)\cdot y, y \text{ is locally of finite variation}\}.
  \end{equation*}
\end{proposition}

We could relax the function $F$ to be locally Lipschitz, but an additional condition on $X^\e$ is required to obtain the LDP, as is shown in the next corollary.

\begin{corollary}\label{local-Lip}
  Let the assumptions in Proposition \ref{LDP-control} hold, except that $F$ is only locally Lipschitz. Assume in addition, for every $t>0$, the family $\{\sup_{0\le s\le t}|X^\e_s|\}_{\e>0}$ is exponentially tight. Then the same results as in Proposition \ref{LDP-control} hold.
\end{corollary}

As we mentioned in the introduction, the rolling process is set up on the orthonormal frame bundle through an SDE of Stratonovich-type. So it is more natural to consider the following controlled SDE of Stratonovich-type:
\begin{equation}\label{SDE-Stra}
  dX^\e_t = F(X^\e_{t},U^\e_{t})\circ dY^\e_t, \quad X^\e_0 = x_0.
\end{equation}
The following LDP result plays an essential role in the next section.
\begin{corollary}\label{exp-tight-Strat}
  Let $F$ be bounded and global Lipschitz with $\nabla F$ also global Lipschitz. Let $X^\e$ be the solution of \eqref{SDE-Stra} for each $\e>0$.
  Assume the family $\{G(Y^\e)_t\}_{\e>0}$ is exponentially tight for each $t>0$, and the family $\{(Y^\e,U^\e,\langle M^{\e}, M^{\e}\rangle)\}_{\e>0}$ satisfies the LDP with a good rate function $I'':\R^n \times \R^l \times (\R^n\otimes\R^n) \to [0,\infty]$. 
  Then

  (i). the family $\{(X^\e,Y^\e,U^\e)\}_{\e>0}$ satisfies the LDP with a good rate function $I$,

  (ii). the rate function $I$ is given by
  \begin{equation*}
    I(x,y,u) =
    \begin{cases}
      \inf\{I''(y,u,q): x = F(x,u)\cdot y + \frac{1}{2}(F^j \otimes\partial_{x_j} F)(x,u)\cdot q, q \text{ is locally of finite variation}\}, & \\
      \qquad y \text{ is locally of finite variation}, & \\
      \infty, \text{otherwise}. &
    \end{cases}
  \end{equation*}
  In particular, the family $\{X^\e\}_{\e>0}$ satisfies the LDP with the following good rate function:
  \begin{equation*}
    I^\flat(x) =\inf\{I''(y,u,q): x = F(x,u)\cdot y + \textstyle{\frac{1}{2}}(F^j \otimes\partial_{x_j} F)(x,u)\cdot q, y  \text{ and } q \text{ are locally of finite variation}\}.
  \end{equation*}
\end{corollary}
\begin{proof}
  We can rewrite the Stratonovich-type SDE as
  \begin{equation*}
    d\binom{X^\e_t}{U^\e_t} = \left(
                                \begin{array}{cc}
                                  F(X^\e_t,U^\e_t) & 0 \\
                                  0 & \I_l \\
                                \end{array}
                              \right) \circ d\binom{Y^\e_t}{U^\e_t},
  \end{equation*}
  where $\I_l$ denotes the $l\times l$ identity matrix. Then using the decomposition $Y^\e = M^\e+A^\e$, we have
  \begin{equation*}
    dX^{\e,i}_t = F^i_j(X^\e_{t},U^\e_{t})dM^{\e,j}_t + F^i_j(X^\e_{t},U^\e_{t})dA^{\e,j}_t + \frac{1}{2} (F_k^j \partial_{x_j} F_h^i)(X^\e_{t},U^\e_{t})d\langle M^{\e,k}, M^{\e,h}\rangle_t,
  \end{equation*}
  that is,
  \begin{equation}\label{SDE-str}
    dX^\e_t = \left( F(X^\e_t,U^\e_t) \quad \textstyle{\frac{1}{2}} (F^j \otimes \partial_{x_j} F)(X^\e_{t},U^\e_{t}) \right) d\binom{Y^\e_t}{\langle M^{\e}, M^{\e}\rangle_t},
  \end{equation}
  where we view $\langle M^{\e}, M^{\e}\rangle$ as an $\R^n\otimes\R^n$-valued process. It is easy to see that the coefficients in \eqref{SDE-str} are global Lipschitz and bounded by assumptions. Using the exponential tightness of $\{\langle M^{\e}, M^{\e}\rangle\}$ and Lemma \ref{exp-tight}.(i), we know that the family $\{|\langle M^{\e}, M^{\e}\rangle|_t\}$ is exponentially tight for each $t>0$. Thus, we can see that Proposition \ref{LDP-control} is applicable here, with $(Y^\e,\langle M^{\e}, M^{\e}\rangle)$ in place of $Y^\e$ and $|V(A^\e)| + (\frac{1}{\e}+1)|\langle M^\e, M^\e \rangle|$ in place of $G(Y^\e)$. The family $\{(X^\e,Y^\e,U^\e,\langle M^{\e}, M^{\e}\rangle)\}_{\e>0}$ satisfies the LDP with the following good rate function
  \begin{equation*}
    \tilde I(x,y,u,q) =
    \begin{cases}
      I''(y,u,q), & x = F(x,u)\cdot y + \frac{1}{2}(F^j \otimes\partial_{x_j} F)(x,u)\cdot q, y \text{ and } q \text{ are locally of finite variation}, \\
      \infty, & \text{otherwise}.
    \end{cases}
  \end{equation*}
  The results follow from the contraction principle.
\end{proof}

\begin{remark}
  The exponential tightness of $\{|\langle M^{\e}, M^{\e}\rangle|_t\}$ also follows from that of $\{\frac{1}{\e}|\langle M^{\e}, M^{\e}\rangle|_t\}$.
\end{remark}

\section{Large deviations for rolling systems}\label{sec-LDP-rolling}

In this section, we will present our main results Theorem \ref{cpct}, Theorem \ref{non-cpct-1} and Theorem \ref{non-cpct-2}, which provide the large deviation principles for various rolling systems.

\subsection{Preliminaries}

Let $M$ be a $d$-dimensional Riemannian manifold equipped with the Levi-Civita connection, $\pi:OM\to M$ be its orthonormal frame bundle with structure group $O(d)$. We denote by $\omega$ the connection form associated with the Levi-Civita connection. Denote by $H_\xi$ the standard horizontal vector field on $OM$ corresponding to $\xi\in\R^d$, uniquely characterized by the property that $\pi_*(H_\xi(u)) = u(\xi)$ for all $u \in OM$. Let $\{e_1,...,e_d\}$ be the canonical basis of $\mathbf R^d$, with dual basis $\{e_1^*,...,e_d^*\}$. We denote in shorthand $H_i:=H_{e_i}$. Then $H_\xi = \xi^i H_i$. Every $A\in \mathfrak{so}(d)$, Lie algebra of the rotation group $SO(d)$, induces a vector field $A^*$ on $OM$, called the fundamental vector field corresponding to $A$. It is well-known that the Lie algebra $\mathfrak{so}(d)$ consists of all skew-symmetric $d\times d$ matrices, and its dimension is $D=d(d-1)/2$. We denote by $a_i^j = e_i\otimes e_j^* - e_j \otimes e_i^*$ the $d\times d$ matrix such that the entry at the $i$-th column and the $j$-th row is $1$, the entry at the $j$-th column and the $i$-th row is $-1$ and other entries are all zero. Then $\{a_i^j \mid 1\le i<j\le d\}$ form a basis of $\mathfrak{so}(d)$. We rearrange this basis into $\{A_1,\cdots, A_D\}$. We refer the reader to \cite{KN63} and \cite{Hsu02} for more details on the differential geometry and its interactions with stochastic analysis.

Recall that the equation of rolling perturbed by small random slipping and twisting is set up on $OM$ as
\begin{equation}\label{rolling}
  d u^\e_t = H_{i}(u^\e_t) \circ d\gamma^{\e,i}_t + A_\alpha^*(u^\e_t)\circ d W_t^{\e,\alpha}, \quad u_0^\e = u_0\in OM,
\end{equation}
Let $x^\e = \{x^\e_t\}_{t\ge0}$ be the projection curve $x^\e_t = \pi(u^\e_t)$. We denote by $\tilde x^\e = \{\tilde x^\e_t\}_{t\ge0}$ the horizontal lift of the curve $\{x^\e_t\}$ starting at $u_0$. Then we have
\begin{lemma}\label{transform}
  Let 
  \begin{equation}\label{u-x-tilde}
    u^\e_t = \tilde x^\e_t g^\e_t,
  \end{equation}
  with $g^\e_t\in SO(d)$ for every $t\ge0$ and $\e>0$. Then
  \begin{align}
    d\tilde x^\e_t &= H_i(\tilde x^\e_t) (g^{\e,i}_j)_t \circ d\gamma^{\e,j}_t, \quad \tilde x^\e_0 = u_0, \label{sde-x-tilde} \\
    d g^\e_t &= g^\e_t A_\alpha \circ dW^{\e,\alpha}_t, \quad g^\e_0 = \I_d, \label{sde-g}
  \end{align}
  where $\I_d$ denotes the $d\times d$ identity matrix.
\end{lemma}
\begin{proof}
  We follow the lines of \cite[Lemma 3.1]{Li16a}. By \eqref{rolling} and $\tilde x^\e = \{\tilde x^\e_t\}_{t\ge0}$,
  \begin{equation*}
    d x^\e_t = \pi^*(du^\e_t) = \pi^*((du^\e_t)^h) = \pi^*(H_{i}(u^\e_t) \circ d\gamma^{\e,i}_t) = u^\e_t\circ d\gamma^\e_t = \tilde x^\e_t g^\e_t\circ d\gamma^\e_t,
  \end{equation*}
  which implies \eqref{sde-x-tilde}. On the other hand, \eqref{u-x-tilde} yields that
  \begin{equation*}
    du^\e_t = d(R_{g^\e_t})(\circ d\tilde x^\e_t) + \tilde x^\e_t \circ dg^\e_t.
  \end{equation*}
  For a $g\in O(d)$, we denote by $R_g$ and $L_g$ the right and left action of $g$ on $OM$, and denote by $d(R_g)$ and $d(R_g)$ their differentials. Then
  \begin{equation*}
    \omega(\circ du^\e_t) = \Ad((g^\e_t)^{-1})\omega(\circ d\tilde x^\e_t) + d(L_{(g^\e_t)^{-1}})(\circ dg^\e_t) = d(L_{(g^\e_t)^{-1}})(\circ dg^\e_t).
  \end{equation*}
  We also have by \eqref{rolling} that
  \begin{equation*}
    \omega(\circ du^\e_t) = A_\alpha \circ dW^{\e,\alpha}_t.
  \end{equation*}
  Hence
  \begin{equation*}
    dg^\e_t = d(L_{g^\e_t})(A_\alpha \circ dW^{\e,\alpha}_t) = g^\e_t A_\alpha \circ dW^{\e,\alpha}_t.
  \end{equation*}
  This gives \eqref{sde-g}.
\end{proof}

We will also need some probabilistic preparation. Given two continuous one-dimensional local martingales $X$ and $Y$ on a filtered probability space $(\Omega,\F,\{\F_t\}_{t\ge0},\P)$, the \emph{quadratic covariation} of $X$ and $Y$ is defined by
\begin{equation*}
  \langle X,Y\rangle_t := \lim_{\|\Delta\| \to 0} \sum_i (X_{t_{i+1}}-X_{t_i})(Y_{t_{i+1}}-Y_{t_i})
\end{equation*}
where $\Delta$ ranges over all partitions $\{t_i\}_{i=0}^k$ of the interval $[0,t]$ with $0=t_0\le t_1\le \cdots \le t_k=t$, and $\|\Delta\| = \max_{1\le i\le k}|t_{i+1}-t_i|$ is the mesh of $\Delta$. This limit, if it exists, is defined using convergence in probability. The following lemma will be used.
\begin{lemma}\label{co-var}
  Let $X$ be a one-dimensional continuous local martingale, and $W$ be a one-dimensional standard Brownian motion independent of $X$. Then we have $\langle X,W\rangle = 0$.
\end{lemma}
\begin{proof}
  Fix $t>0$. By the definition of quadratic covariation, it is sufficient to prove that
  \begin{equation*}
    \sum_i (X_{t_{i+1}}-X_{t_i})(W_{t_{i+1}}-W_{t_i}) \to 0, \text{ in } L^2, \quad\text{as } \|\Delta\| \to 0.
  \end{equation*}
  We denote by $S^\Delta$ the summation in the previous sentence. Then
  \begin{equation*}
    \begin{split}
      |S^\Delta|^2 =&\ \sum_i (X_{t_{i+1}}-X_{t_i})^2(W_{t_{i+1}}-W_{t_i})^2 \\
         &\ + \sum_{i\ne j} (X_{t_{i+1}}-X_{t_i})(X_{t_{j+1}}-X_{t_j}) (W_{t_{i+1}}-W_{t_i})(W_{t_{j+1}}-W_{t_j}) \\
         =:&\ I_1 + I_2.
    \end{split}
  \end{equation*}
  For $I_1$, using the independence of $X$ and $W$, we have
  \begin{equation*}
    \E(I_1) = \sum_i (t_{i+1}-t_i) \E(X_{t_{i+1}}-X_{t_i})^2 \le \|\Delta\| \sum_i \E(X_{t_{i+1}}-X_{t_i})^2,
  \end{equation*}
  which goes to zero as $\|\Delta\|\to 0$, since $\sum_i \E(X_{t_{i+1}}-X_{t_i})^2\to \E(\langle X,X\rangle)$. For $I_2$, since $W$ has independent increments,
  \begin{equation*}
    \E(I_2) = \sum_{i\ne j} \E [(X_{t_{i+1}}-X_{t_i})(X_{t_{j+1}}-X_{t_j})] \E (W_{t_{i+1}}-W_{t_i}) \E(W_{t_{j+1}}-W_{t_j}) = 0.
  \end{equation*}
  Combining these, $\E|S^\Delta|^2\to0$ as $\|\Delta\|\to 0$, the result follows.
\end{proof}

We denote the canonical decomposition of the continuous semimartingale $\gamma^\e$ by
\begin{equation*}
  \gamma^\e = m^\e + a^\e,
\end{equation*}
where $m^\e$ is the continuous local martingale part, and $a^\e$ is the continuous finite variation part.

We define the space $\H^1$ by
\begin{equation*}
  \H^1 := \H^1(\R_+,\R^d) = \{ f: \R_+ \to \R^d \mid f \text{ is absolutely continuous and } \textstyle{\int_0^\infty} |\dot f(t)|^2 dt < \infty \},
\end{equation*}
equipped with norm $\|f\|_{\H^1} = \int_0^\infty |\dot f(t)|^2 dt$.

\subsection{For compact manifolds}\label{sec-cpct}

The manifold rolled against the Euclidean space is assumed to be compact in this subsection. The random Euclidean curves that the manifold rolls along are general continuous semimartingales.

The following theorem is the first main result of this paper.

\begin{theorem}\label{cpct}
  Let $M$ be a compact Riemannian manifold. Then for each $\e>0$, the SDE \eqref{rolling} is conservative. If the family $\{(\gamma^\e,\langle m^\e,m^\e\rangle)\}_{\e>0}$ satisfies the LDP with a good rate function $I$, and the family $\{G(\gamma^\e)_t\}_{\e>0}$
  is exponentially tight for each $t> 0$, then $\{u^\e\}_{\e>0}$, $\{\tilde x^\e\}_{\e>0}$ and $\{x^\e\}_{\e>0}$ all satisfy the LDP with the following good rate functions respectively:
  \begin{align}
    I_{OM}(u) &= \inf\left\{I(y,q) + \textstyle{\frac{1}{2}\int_0^\infty} |\dot f(t)|^2dt :
    \begin{array}{l}
    u = u_0 + H_i(u)\cdot y^i + A^*_\alpha(u)\cdot f^\alpha \\
    \qquad + \frac{1}{2}(H_k H_h)(u) \cdot q^{kh}, \\
    f\in \H^1, y \text{ and } q \text{ are locally of finite variation}
    \end{array}
    \right\}, \label{rate-u} \\
    \tilde I_{OM}(v) &= \inf\left\{ I(y,q) + \textstyle{\frac{1}{2}\int_0^\infty} |\dot f(t)|^2dt :
    \begin{array}{l}
      v= u_0+H_i(v)g_j^i\cdot y^j + \frac{1}{2}(H_i H_j)(v)g^i_k g^j_h \cdot q^{kh}, \\
      g = \I_d + g A_\alpha \cdot f^\alpha, f\in \H^1, \\
      y \text{ and } q \text{ are locally of finite variation}
    \end{array}
    \right\}, \label{rate-tilde-x} \\
    I_M(x) &= \inf \{ I_{OM}(u): \pi(u) = x \} = \inf \{ \tilde I_{OM}(v): \pi(v) = x \}. \label{rate-x}
  \end{align}
\end{theorem}
\begin{proof}
  The well-known Whitney's embedding theorem \cite[Theorem 6.15]{Lee12} implies that $OM$ admits a proper smooth embedding into $\R^p$ with some $p\in\N$, which we denote as $\iota: OM\to \R^p$. Using this embedding, the rolling equation becomes
  \begin{equation*}
    d (\iota \circ u^\e)_t = \iota_*(H_{i})((\iota \circ u^\e)_t) \circ d\gamma^{\e,i}_t + \iota_*(A_\alpha^*)((\iota \circ u^\e)_t)\circ d W_t^{\e,\alpha},
  \end{equation*}
  Since $M$ is compact, its orthonormal frame bundle $OM$ with structure group $O(d)$ is also compact. Thus, the smooth vector fields $\iota_*(H_i)$'s and $\iota_*(A^*_\alpha)$'s on $\iota(OM)$ are all bounded and Lipschitz, with Lipschitz derivatives. The properness of the embedding $\iota$ yields that $\iota(OM)$ is a closed submanifold of $\R^p$. Thus, one can extend each $\iota_*(H_i)$ and $\iota_*(A^*_\alpha)$ to bounded and Lipschitz vector fields on $\R^p$ with Lipschitz derivatives, which we denote as $\widetilde{\iota_*(H_i)}$ and $\widetilde{\iota_*(A^*_\alpha)}$. Since these extended vector fields are tangent to $\iota(OM)$ when restricted on it, $\iota \circ u^\e$ is also the solution to the following equation
  \begin{equation}\label{embed-sde}
    d (\iota \circ u^\e)_t = \widetilde{\iota_*(H_{i})}((\iota \circ u^\e)_t) \circ d\gamma^{\e,i}_t + \widetilde{\iota_*(A_\alpha^*)}((\iota \circ u^\e)_t)\circ d W_t^{\e,\alpha}.
  \end{equation}
  Then the classical well-posedness theory of SDEs (see, e.g., \cite[Theorem 5.2.5, 5.2.9]{KS91}) yields that the solution $\iota \circ u^\e$ of \eqref{embed-sde} globally exists (without finite-time explosion) and is unique. Since $\iota \circ u^\e$ starts from $\iota(OM)$ and the vector fields $\widetilde{\iota_*(H_i)}$ and $\widetilde{\iota_*(A^*_\alpha)}$ are tangent to $\iota(OM)$ on $\iota(OM)$, it is expected that $\iota \circ u^\e$ never leaves $\iota(OM)$ (cf. \cite[Proposition 1.2.8]{Hsu02}). Since $\iota$ is smooth embedding, the conservation of \eqref{rolling} follows.

  Since $\langle W^\e,W^\e\rangle_t = \e t\I_D$, the family $\langle W^\e,W^\e\rangle$ converges to the constant path $0$ in $\C^{D\times D}$ as $\e\to0$. It follows from the classical results that the family $\{(W^\e,\langle W^\e,W^\e\rangle)\}$ satisfies the LDP with the good rate function
  \begin{equation*}
    I'(f,g) =
    \begin{cases}
      \frac{1}{2}\int_0^\infty|\dot f(t)|^2dt, & \mbox{if } f\in \H^1, g \equiv 0, \\
      \infty, & \mbox{otherwise}.
    \end{cases}
  \end{equation*}
  The independence of $\gamma^\e$ and $W^\e$ yields that the LDP holds for $\{(\gamma^\e,\langle m^\e,m^\e\rangle,W^\e,\langle W^\e,W^\e\rangle)\}$ with the good rate function $I''(y,q,f,g) = I(y,q)+ I'(f,g)$. It follows from Lemma \ref{co-var} that
  \begin{equation*}
    \langle (m^\e,W^\e),(m^\e,W^\e)\rangle = \left(
                                               \begin{array}{cc}
                                                 \langle m^\e,m^\e\rangle & 0 \\
                                                 0 & \langle W^\e,W^\e\rangle \\
                                               \end{array}
                                             \right).
  \end{equation*}
  Due to the contraction principle, $\{(\gamma^\e,W^\e,\langle (m^\e,W^\e),(m^\e,W^\e)\rangle )\}$ satisfies LDP with the good rate function
  \begin{equation*}
    I'''(y,f,Q) =
    \begin{cases}
      I''(y,Q_{11},f,Q_{22}), & \mbox{if } Q_{21} = 0 \text{ and } Q_{12} = 0, \\
      \infty, & \mbox{otherwise},
    \end{cases}
  \end{equation*}
  where $Q = \left( \begin{array}{cc}
                 Q_{11} & Q_{12} \\
                 Q_{21} & Q_{22}
               \end{array} \right)$
  is a $(d+D)\times(d+D)$ matrix with $Q_{11}$, $Q_{22}$ the $d\times d$ and $D\times D$ submatrix respectively, etc. On the other hand, $G((\gamma^\e,W^\e))_t = G(\gamma^\e)_t + \frac{1}{\e}|\langle W^\e,W^\e\rangle|_t = G(\gamma^\e)_t + Dt$. For each $t$, the exponential tightness of $\{G(\gamma^\e)_t\}$ yields that $\{G((\gamma^\e,W^\e))_t\}$ is also exponentially tight.

  Now we can apply Corollary \ref{exp-tight-Strat}, the family $\{\iota\circ u^\e\}$ satisfies the LDP with the good rate function
  \begin{equation*}
    \begin{split}
      I_{\iota(OM)}(x) & = \inf\{I''(y,q,f,g): x = \iota(u_0) + \widetilde{\iota_*(H_i)}(x)\cdot y^i + \widetilde{\iota_*(A^*_\alpha)}(x)\cdot f^\alpha+ \textstyle{\frac{1}{2}}(\widetilde{\iota_*(H_k)}^j \partial_{x_j}\widetilde{\iota_*(H_h)})(x) \cdot q^{kh} \\
      &\qquad\quad + \textstyle{\frac{1}{2}}(\widetilde{\iota_*(A^*_\alpha)}^j \partial_{x_j}\widetilde{\iota_*(A^*_\beta)})(x) \cdot g^{\alpha\beta}, y,q,f \text{ and } g \text{ are locally of finite variation} \} \\
      & = \inf\{I(y,q) + \textstyle{\frac{1}{2}\int_0^\infty} |\dot f(t)|^2dt : x = \iota(u_0) + \widetilde{\iota_*(H_i)}(x)\cdot y^i + \widetilde{\iota_*(A^*_\alpha)}(x)\cdot f^\alpha \\
      &\qquad\quad + \textstyle{\frac{1}{2}}(\widetilde{\iota_*(H_k)}^j \partial_{x_j}\widetilde{\iota_*(H_h)})(x) \cdot q^{kh}, f\in \H^1, y \text{ and } q \text{ are locally of finite variation} \} \\
      & = \inf\{I(y,q) + \textstyle{\frac{1}{2}\int_0^\infty} |\dot f(t)|^2dt : x = \iota(u_0) + \iota_*(H_i)(x)\cdot y^i + \iota_*(A^*_\alpha)(x)\cdot f^\alpha \\
      &\qquad\quad + \textstyle{\frac{1}{2}}(\iota_*(H_k)^j \partial_{x_j}\iota_*(H_h))(x) \cdot q^{kh}, f\in \H^1, y \text{ and } q \text{ are locally of finite variation} \},
    \end{split}
  \end{equation*}
  where the last equality follows from the fact that the path $x$ always stays on $\iota(OM)$, since all the vector fields $\widetilde{\iota_*(H_i)}$ and $\widetilde{\iota_*(A^*_k)}$ are tangent to $\iota(OM)$ when restricted on it and the initial value $\iota(u_0)\in\iota(OM)$. Since $\iota$ is a smooth embedding, the inverse contraction principle (see, e.g., \cite[Theorem 4.2.4]{DZ98}) implies that the family $\{u^\e\}$ satisfies the LDP with the good rate function
  \begin{equation*}
    I_{OM}(u) = I_{\iota(OM)}(\iota\circ u),
  \end{equation*}
  which yields \eqref{rate-u}.

  Now we turn to $\{\tilde x^\e\}$. 
  We will use Lemma \ref{transform}. It is easy to see that the coefficient function of SDE \eqref{sde-g} is global Lipschitz and bounded. The classical Freidlin-Wentzell theory for Stratonovich-type SDE (see \cite[Theorem 2.5]{KRV18}, or using our general result Corollary \ref{exp-tight-Strat}) yields that the family $\{g^\e\}$ satisfies the LDP with the good rate function
  \begin{equation*}
    \tilde I'(g) = \inf\{\textstyle{\frac{1}{2}} \int_0^\infty|\dot f(t)|^2dt: g = \I_d + g A_\alpha \cdot f^\alpha, f\in \H^1 \}.
  \end{equation*}
  Since $\gamma^\e$ is independent of $W^\e$, it is also independent of $g^\e$. Thus the family $\{(\gamma^\e,\langle m^\e,m^\e\rangle,g^\e)\}$ satisfies the LDP with the good rate function $\tilde I''(y,q,g) = I(y,q) + \tilde I'(g)$. Similar as $u^\e$, Corollary \ref{exp-tight-Strat} yields that the family $\{\tilde x^\e\}$ satisfies the LDP with the good rate function
  \begin{equation*}
    \tilde I_{OM}(v) = \inf\{ \tilde I''(y,q,g): v= u_0+H_i(v)g_j^i\cdot y^j + \textstyle{\frac{1}{2}}(H_i H_j)(v)g^i_k g^j_h \cdot q^{kh}, y \text{ and } q \text{ are locally of finite variation} \},
  \end{equation*}
  the representation \eqref{rate-tilde-x} follows. The large deviation principle of $\{x^\e\}$ and its rate function \eqref{rate-x} follows from the continuity of the projection $\pi$ and the contraction principle.
\end{proof}

\subsection{For noncompact manifolds}\label{sec-noncpct}

In this subsection, we deal with the rolling systems where the rolled manifold $M$ is a noncompact Riemannian manifold. For the Euclidean curves, we only consider two special case: the random curves with locally finite variation and the random curve driven by stochastic differential equations.

\subsubsection{Along pathwise rectifiable random curves}

We consider the case that $m^\e \equiv 0$, that is, each $\gamma^\e$ is a continuous process with locally finite variation. In this case, both Stratonovich circles in \eqref{rolling} and \eqref{sde-x-tilde} can be removed, because the integrals therein are in the Riemann-Stieltjes sense. Note also that in this case, $G(\gamma^\e) = V(\gamma^\e)$.

\begin{theorem}\label{non-cpct-1}
  Let $M$ be a geodesically complete Riemannian manifold. Let $\{\gamma^\e\}_{\e>0}$ be a family of adapted continuous processes with locally finite variation. Then for each $\e>0$, the SDE \eqref{rolling} is conservative. If the family $\{\gamma^\e\}_{\e>0}$ satisfies the LDP with a good rate function $I^\flat$, and the family $\{|V(\gamma^\e)|_t \}_{\e>0}$
  is exponentially tight for each $t> 0$, then $\{u^\e\}_{\e>0}$ and $\{\tilde x^\e\}_{\e>0}$ both satisfy the LDP with the following good rate functions respectively:
  \begin{align}
    I_{OM}(u) &= \inf\left\{ I^\flat(y) + \textstyle{\frac{1}{2}\int_0^\infty} |\dot f(t)|^2dt :
    \begin{array}{l}
      u = u_0 + H_i(u)\cdot y^i + A^*_\alpha(u)\cdot f^\alpha, \\
      f\in \H^1, y \text{ is locally of finite variation}
    \end{array}
    \right\}, \label{rate-u-1} \\
    \tilde I_{OM}(v) &= \inf\left\{ I^\flat(y) + \textstyle{\frac{1}{2}\int_0^\infty} |\dot f(t)|^2dt :
    \begin{array}{l}
    v = u_0 + H_i(v)g_j^i\cdot y^j, g = \I_d + g A_\alpha \cdot f^\alpha, \\
    f\in \H^1, y \text{ is locally of finite variation}
    \end{array}
    \right\}, \label{rate-tilde-x-1}
  \end{align}
  and the family $\{x^\e\}_{\e>0}$ satisfies the LDP with the good rate function \eqref{rate-x}.
\end{theorem}

\begin{proof}
  Fix $\e>0$. Denote by $T^\e$ the lifetime of $u^\e$. Assume, by contradiction, $\P^\e(T^\e<\infty) >0$, that is, $u^\e$ would explode at finite time $T^\e$ with positive probability. Since $\gamma^\e$ is locally of finite variation, each sample path of it is rectifiable, and for each $t>0$, the length of $\gamma^\e|_{[0,t]}$ is $V(\gamma^\e)_t$. Then by \eqref{sde-x-tilde}, for every $t>0$,
  \begin{equation*}
    L(x^\e|_{[0,t]}) = L\left( \left\{\int_0^r g^\e_s d\gamma^\e_s: r \in[0,t] \right\} \right) = \int_0^t |g^\e_s| dV(\gamma^\e)_s = V(\gamma^\e)_t.
  \end{equation*}
  That is, $x^\e$ is also rectifiable. Thus, the completeness of $M$ yields that the continuous path $x^\e|_{[0,T^\e]}$ converges as $t\to T^\e$, on the event $\{T^\e<\infty\}$. As a consequence, the horizontal lift $\tilde x^\e|_{[0,T^\e]}$ also converge as $t\to T^\e$. Since the equation \eqref{sde-g} can be globally solved, by \eqref{u-x-tilde}, the process $u^\e|_{[0,T^\e]}$ would converge as well, as $t\to T^\e$. This leads to a contradiction with the necessary explosion of $u^\e$. Therefore, the SDE \eqref{rolling} is conservative.

  For the second statement, we note that we have identified the rate functions in Theorem \ref{cpct}. Since each $\gamma^\e$ is locally of finite variation, $m^\e \equiv 0$ and the family $\{(\gamma^\e,\langle m^\e,m^\e\rangle)\}$ satisfies the LDP with the good rate function
  \begin{equation*}
    I(y,q) =
    \begin{cases}
      I^\flat(y), & \mbox{if } q \equiv 0, \\
      \infty, & \mbox{otherwise}.
    \end{cases}
  \end{equation*}
  Thus, the representation \eqref{rate-u-1} and \eqref{rate-tilde-x-1} follow from \eqref{rate-u} and \eqref{rate-tilde-x} respectively. Now it is only needed to prove that the family $\{u^\e\}$ and $\{\tilde x^\e\}$ are exponentially tight. We will use a similar argument as Corollary \ref{local-Lip}.

  Suppose first that there exists a constant $R>0$ such that $d(x_0,x^\e_t) \le R$ for all $t>0$, where $d$ is the Riemannian distance function on $M$. Let $K_R := \overline{B_M(x_0,R)}$. The completeness of $M$ implies $K_R = \exp(B_{T_{x_0}M}(0,R))$. Thus, $K_R$ is a compact submanifold of $M$. Since all the paths of $x^\e$ are contained in the submanifold $K_R$, all paths of $u^\e$ and $\tilde x^\e$ are contained in the bundle $OM|_{K_R} = O(K_R)$. Consequently, the SDEs \eqref{rolling} and \eqref{sde-x-tilde} can be viewed as equations on $O(K_R)$. The exponential tightness of $\{u^\e\}$ and $\{\tilde x^\e\}$ follows from Theorem \ref{cpct}.

  Now we prove the general case. For each $\e>0$ and $p>0$, define a stopping time
  \begin{equation*}
    T^{\e,p} = \inf\{ t\ge0: d(x_0,x^\e_t) \ge p \}.
  \end{equation*}
  Observe that
  \begin{equation*}
    d(x_0,x^\e_t) \le L(x^\e|_{[0,t]}) = V(\gamma^\e)_t.
  \end{equation*}
  Then
  \begin{equation*}
    \P^\e(T^{\e,p} \le T) = \P^\e\left( d(x_0,x^\e_t) \ge p \right) \le \P^\e\left( V(\gamma^\e)_t \ge p \right).
  \end{equation*}
  Then a similar localization argument in Step 2 of the proof of Proposition \ref{LDP-control} yields the exponential tightness.
\end{proof}

\subsubsection{Along random curves defined by stochastic differential equations}

For each $\e>0$, let $\gamma^\e$ be the solution of the following SDE:
\begin{equation}\label{random-pert}
  d\gamma^\e_t = b(t,\gamma^\e_t) dt + \sqrt\e dB_t, \quad \gamma^\e_0 = \gamma_0.
\end{equation}
where $B$ is a $d$-dimensional Brownian motion independent with $W$, $b: [0,\infty)\times\R^d \to\R^d$ is a function satisfying
\begin{equation*}
  |b(t,x)| \le C, \quad |b(t,x_1) - b(t,x_2)| \le C|x_1-x_2|,
\end{equation*}
for all $t\ge0$ and $x,x_1,x_2\in\R^d$ with some constant $C>0$. It is well-known that the SDE \eqref{random-pert} has a unique strong solution.

\begin{theorem}\label{non-cpct-2}
  Let $M$ be a 
  stochastically complete Riemannian manifold and assume there exists a constant $L>1$, such that the Ricci curvature is bounded below by $-L$. For each $\e>0$, let $\gamma^\e$ be the unique solution of \eqref{random-pert}. Then for each $\e>0$, the SDE \eqref{rolling} is conservative, and the families $\{u^\e\}_{\e>0}$ and $\{\tilde x^\e\}_{\e>0}$ both satisfy the LDP with the following good rate functions respectively:
  \begin{align}
    I_{OM}(u) &= \inf\left\{ \textstyle{\frac{1}{2}\int_0^\infty} (|\dot y(t) - b(t,y(t))|^2+|\dot f(t)|^2 ) dt:
    \begin{array}{l}
    u = u_0 + H_i(u)\cdot y^i + A^*_\alpha(u)\cdot f^\alpha, \\
    y,f\in \H^1
    \end{array}
    \right\}, \label{rate-u-2} \\
    \tilde I_{OM}(v) &= \inf\left\{ \textstyle{\frac{1}{2}\int_0^\infty} \left(|\dot y(t) - b(t,y(t))|^2+|\dot f(t)|^2 \right) dt:
    \begin{array}{l}
      v = u_0 + H_i(v)g_j^i\cdot y^j, \\
      g = \I_d + g A_\alpha \cdot f^\alpha, \\
      y,f\in \H^1
    \end{array}
    \right\}, \label{rate-tilde-x-2}
  \end{align}
  and the family $\{x^\e\}_{\e>0}$ satisfies the LDP with the good rate function \eqref{rate-x}.
\end{theorem}
\begin{proof}
  Firstly, we show the non-explosion of \eqref{rolling}. For each $\e>0$, define $\zeta^\e$ to be the following integral process
  \begin{equation*}
    \zeta^\e_t := \int_0^t g^\e_s \circ d\gamma^\e_s = \int_0^t g^\e_s b(s,\gamma^\e_s) ds + \sqrt\e\int_0^t g^\e_s \circ dB_s = \int_0^t g^\e_s b(s,\gamma^\e_s) ds + \sqrt\e\int_0^t g^\e_s dB_s.
  \end{equation*}
  Here the last equality follows from $\langle g^\e, B \rangle = 0 $ by \eqref{sde-g} and the independence. Since each process $g^\e$ takes values in $O(d)$, we calculate the quadratic variation
  $$\left\langle \int_0^\cdot g^\e_s dB_s, \int_0^\cdot g^\e_s dB_s \right\rangle_t = \int_0^t g^\e_s (g^\e_s)^* ds = t \I_d.$$
  By L\'evy's characterization of Brownian motion (e.g., \cite[Theorem 3.3.16]{KS91}), the stochastic integral $\int_0^\cdot g^\e_s dB_s$ is a $d$-dimensional standard Brownian motion under $\P^\e$, which we denote as $\hat A^\e$. Then
  \begin{equation}\label{zeta}
    \zeta^\e_t = \int_0^t g^\e_s b(s,\gamma^\e_s) ds + \sqrt\e \hat A^\e_t.
  \end{equation}
  By \eqref{sde-x-tilde} and the definition of $\zeta^\e$, we have
  \begin{equation}\label{sde-rewrite}
    d\tilde x^\e_t = H_i(\tilde x^\e_t) \circ d\zeta^{\e,i}_t.
  \end{equation}
  Define
  \begin{equation}\label{Girsanov}
    Z^\e_t = \exp\left\{ -\frac{1}{\sqrt\e} \int_0^t g^\e_s b(s,\gamma^\e_s) d\hat A^\e_s -\frac{1}{2\e} \int_0^t |g^\e_s b(s,\gamma^\e_s)|^2 ds \right\}.
  \end{equation}
  Then the boundedness of $b$ and $g^\e$ yields that $Z^\e$ is a continuous martingale. Define, for each $t>0$, a probability measure on $\F^\e_t$ by $\tilde\P^\e_t(A) := \E^\e (Z_t^\e;A)$. By Girsanov theorem \cite[Theorem 3.5.1]{KS91}, each process $\{\zeta^\e_s\}_{0\le s \le t}$ in \eqref{zeta} is a Brownian motion with covariance matrix $\e \I_d$ under $\tilde\P^\e_t$, or equivalently, $\{\zeta^\e_{s/\e}\}_{0\le s \le t}$ is a standard Brownian motion under $\tilde\P^\e_t$.

  Denote by $T^\e$ the lifetime of each $x^\e$. Then for each $\e>0$ and $t>0$, the lifetime of the process $\{x^\e_{s/\e}\}_{0\le s \le t}$ is $(\e T^\e)\wedge t$. By \eqref{sde-rewrite}, we know that the process $\{x^\e_{s/\e}\}_{0\le s \le t}$ is a standard Riemannian Brownian motion under $\tilde\P^\e_t$. Since $M$ is stochastically complete, we have
  \begin{equation*}
    \tilde\P^\e_t( T^\e \ge t/\e ) = \tilde\P^\e_t( (\e T^\e)\wedge t = t ) =1.
  \end{equation*}
  Meanwhile, by \eqref{Girsanov} and \eqref{zeta}, it is easy to deduce
  \begin{equation*}
    (Z_t^\e)^{-1} = \exp\left\{ \frac{1}{\e} \int_0^t g^\e_s b(s,\gamma^\e_s) d\zeta^\e_s -\frac{1}{2\e} \int_0^t |g^\e_s b(s,\gamma^\e_s)|^2 ds \right\}.
  \end{equation*}
  Then $\{(Z_s^\e)^{-1}\}_{0\le s\le t}$ is a continuous martingale under $\tilde\P^\e_t$ with $\tilde\E^\e_t((Z_t^\e)^{-1}) = \tilde\E^\e_t((Z_0^\e)^{-1}) = 1$, since $\{\zeta^\e_s\}_{0\le s \le t}$ is a Brownian motion with covariance matrix $\e \I_d$ under this probability measure. Hence, for every $t>0$,
  \begin{equation*}
    \P^\e( T^\e \ge t/\e ) = \tilde\E^\e_t\left( (Z^\e_t)^{-1}; T^\e \ge t/\e \right) = \tilde\E^\e_t\left( (Z^\e_t)^{-1} \right) = 1.
  \end{equation*}
  This lead to $\P^\e( T^\e =\infty ) = 1$, which means for each $\e>0$, with probability 1 the process $x^\e$ does not explode. The horizontal lifts $\tilde x^\e$ will also not explode. Since $u^\e = \tilde x^\e g^\e$ with $SO(d)$-valued process $g^\e$ globally defined by Lemma \ref{transform}, the process $u^\e$ does not explode as well, for every $\e>0$.

  We now identify the rate functions. The canonical decomposition of each semimartingale $\gamma^\e$ is $\gamma^\e = a^\e + m^\e$, with $a^\e_t = \int_0^t b(s,\gamma^\e_s) ds$ and $m^\e_t = \sqrt\e B_t$. Then, $\langle m^\e,m^\e\rangle = \e t \I_d$ and
  \begin{equation*}
    G(\gamma^\e)_t = V(a^\e)_t + t = \int_0^t |b(s,\gamma^\e_s)| ds +t \le (C+1)t.
  \end{equation*}
  It follows that $\{G(\gamma^\e)_t\}$ is exponentially tight for each $t>0$. By the classical Freidlin-Wentzell theory and the fact the deterministic path $\langle m^\e,m^\e\rangle$ converges to zero in $\C^{d\times d}$ as $\e\to0$, the family $\{(\gamma^\e,\langle m^\e,m^\e\rangle)\}$ satisfies the LDP with the good rate function
  \begin{equation*}
    I(y,q) =
    \begin{cases}
      \frac{1}{2}\int_0^\infty|\dot y(t) - b(t,y(t))|^2dt, & \mbox{if } y\in \H^1,q=0, \\
      \infty, & \mbox{otherwise}.
    \end{cases}
  \end{equation*}
  The representation \eqref{rate-u-2} and \eqref{rate-tilde-x-2} follow from \eqref{rate-u} and \eqref{rate-tilde-x} respectively.

  To prove the exponential tightness of $\{u^\e\}$ and $\{\tilde x^\e\}$, as in the proof of Theorem \ref{non-cpct-1}, it is enough to show that for each $t>0$, the family $\{\sup_{0\le s\le t} d(x_0,x^\e_s)\}_{\e>0}$ is exponentially tight. Suppose first $b\equiv 0$. Then $d\tilde x^\e_t = H_i(\tilde x^\e_t) \circ d(\sqrt\e \hat B^{\e,i}_t) = H_i(\tilde x^\e_t) \circ d( \hat B^{\e,i}_{\e t})$, that is, the rescaled process $x^\e_{\cdot/\e}$ is a standard Riemannian Brownian motion. By \cite[Proposition 3.7]{KRV18}, we have
  \begin{equation*}
    \P^\e\left( \sup_{0\le s\le t} d(x_0,x^\e_s) \ge a \right) = \P^\e\left( \sup_{0\le s\le \e t} d(x_0,x^\e_{s/\e}) \ge a \right) \le 2 \exp \left\{ -\frac{(kL\e t - \frac{1}{2}a^2)^2}{2a^2 \e t} \right\}.
  \end{equation*}
  The exponential tightness follows. For the general case, we apply Girsanov's transform. Since $\{x^\e_{s/\e}\}_{0\le s \le t}$ is a standard Riemannian Brownian motion under $\tilde\P^\e_t$, we have
  \begin{equation}\label{est-10}
    \tilde\P^\e_t\left( \sup_{0\le s\le t} d(x_0,x^\e_s) \ge a \right) \le 2 \exp \left\{ -\frac{(kL\e t - \frac{1}{2}a^2)^2}{2a^2 \e t} \right\}.
  \end{equation}
  Hence, by the definition of $\tilde\P^\e_t$ and \eqref{est-10}, for any $A>0$,
  \begin{equation}\label{est-11}
    \begin{split}
      &\ \P^\e\left( \sup_{0\le s\le t} d(x_0,x^\e_s) \ge a \right) = \tilde\E^\e_t\left( (Z_t^\e)^{-1}; \sup_{0\le s\le t} d(x_0,x^\e_s) \ge a \right) \\
      =&\ \tilde\E^\e_t\left( (Z_t^\e)^{-1}; \sup_{0\le s\le t} d(x_0,x^\e_s) \ge a, (Z_t^\e)^{-1} < e^{A/\e}\right) + \tilde\E^\e_t\left( (Z_t^\e)^{-1}; \sup_{0\le s\le t} d(x_0,x^\e_s) \ge a, (Z_t^\e)^{-1} \ge e^{A/\e}\right) \\
      \le&\ e^{A/\e}\tilde\P^\e_t\left( \sup_{0\le s\le t} d(x_0,x^\e_s) \ge a \right) + \tilde\E^\e_t\left( (Z_t^\e)^{-1}; (Z_t^\e)^{-1} \ge e^{A/\e}\right) \\
      \le&\ 2 \exp \left\{ \frac{A}{\e}-\frac{(kL\e t - \frac{1}{2}a^2)^2}{2a^2 \e t} \right\} + \tilde\E^\e_t\left( (Z_t^\e)^{-1}; (Z_t^\e)^{-1} \ge e^{A/\e}\right).
    \end{split}
  \end{equation}
  For the second term in the last inequality, we use H\"older's inequality and Chebyshev's inequality to estimate,
  \begin{equation}\label{est-19}
    \begin{split}
      &\ \tilde\E^\e_t\left( (Z_t^\e)^{-1}; (Z_t^\e)^{-1} \ge e^{A/\e}\right) \\
      \le&\ \left[ \tilde\E^\e_t\left( (Z_t^\e)^{-2} \right) \right]^{1/2} \left[ \tilde\P^\e_t\left( (Z_t^\e)^{-1} \ge e^{A/\e}\right) \right]^{1/2} \\
      \le&\ \left[ \tilde\E^\e_t\left( \exp\left\{ \frac{2}{\e} \int_0^t g^\e_s b(s,\gamma^\e_s) d\zeta^\e_s -\frac{1}{\e} \int_0^t |g^\e_s b(s,\gamma^\e_s)|^2 ds \right\} \right) \right]^{1/2} e^{-A/(2\e)} \\
      \le&\ \bigg[ \tilde\E^\e_t\left( \exp\left\{ \frac{4}{\e} \int_0^t g^\e_s b(s,\gamma^\e_s) d\zeta^\e_s -\frac{8}{\e} \int_0^t |g^\e_s b(s,\gamma^\e_s)|^2 ds \right\} \right) \\
      &\ \times \tilde\E^\e_t\left(\exp\left\{ \frac{6}{\e} \int_0^t |g^\e_s b(s,\gamma^\e_s)|^2 ds \right\} \right) \bigg]^{1/4} e^{-A/(2\e)} \\
      \le&\ \exp\left\{ \frac{3C^2t}{2\e} - \frac{A}{2\e} \right\},
    \end{split}
  \end{equation}
  where in the last inequality we used the fact that $\{\exp\{ \frac{4}{\e} \int_0^r g^\e_s b(s,\gamma^\e_s) d\zeta^\e_s -\frac{8}{\e} \int_0^r |g^\e_s b(s,\gamma^\e_s)|^2 ds \} \}_{0\le r\le t}$ is a continuous martingale under $\tilde\P^\e_t$. Therefore, combining \eqref{est-11} and \eqref{est-19}, we have
  \begin{equation*}
    \limsup_{\e\to0} \e\log\P^\e\left( \sup_{0\le s\le t} d(x_0,x^\e_s) \ge a \right) \le \left( A-\frac{a^2}{8t} \right) \vee \left( \frac{3C^2t}{2} - \frac{A}{2} \right).
  \end{equation*}
  The exponential tightness of $\{\sup_{0\le s\le t} d(x_0,x^\e_s)\}_{\e>0}$ follows by letting first $a\to\infty$ and then $A\to\infty$.
\end{proof}

\section{Examples}\label{sec-exmp}

Throughout this section, we fix a differentiable curve $\gamma:[0,\infty)\to \R^d$ satisfying $|\dot\gamma_t|\le C$ for all $t\ge0$ with some constant $C>0$. Then obviously $\gamma$ is rectifiable. We let $M$ be a geodesically and stochastically complete Riemannian manifold.

\begin{example}[Random perturbation]
  A simple case is the Euclidean curve which the manifold $M$ rolls along is a random perturbation of a given curve. To be precise, consider the following family of random perturbation of $\gamma$:
  \begin{equation*}
    \gamma^\e_t = \gamma_t + \sqrt\e B_t.
  \end{equation*}
  Then $\gamma^\e$ solves the SDE \eqref{random-pert} with $b \equiv \dot\gamma$. And Theorem \ref{non-cpct-2} is applicable here.
  \qed
\end{example}

In the next examples, we will study the rolling mode along pathwise rectifiable random perturbation. We will make sure Theorem \ref{non-cpct-1} is applicable, by checking that the family $\{\gamma^\e\}_{\e>0}$ satisfies the LDP with some good rate function, and the family $\{|V(\gamma^\e)|_t \}_{\e>0}$ is exponentially tight for each $t> 0$. The following lemma is a criterion for large deviation principle for the processes with locally finite variation, which is adapted from \cite[Theorem 2.1]{Puh94}.
\begin{lemma}\label{LDP-char}
  Let $\{A^\e\}_{\e>0}$ be a family of continuous adapted processes with locally finite variation. Let $g: [0,\infty) \to \R^d$ be a Borel function such that $\int_0^\infty |g(t)|dt <\infty$. Let $a_t = \int_0^t g(s) ds$. If for every $T>0$ and any $\eta>0$,
  \begin{equation*}
    \limsup_{\e\to0} \e\log \P^\e\left( \sup_{0\le t\le T} |A^\e_t - a_t| \ge \eta \right) = -\infty.
  \end{equation*}
  Then the family $\{A^\e\}_{\e>0}$ satisfies the LDP with the good rate function
  \begin{equation*}
    I(f) =
    \begin{cases}
      \int_0^\infty \sup_{\lambda\in\R^d}\langle\lambda, \dot f(t)- g(t)\rangle dt, & f \text{ is absolutely continuous}, \\
      \infty, & \text{otherwise}.
    \end{cases}
  \end{equation*}
\end{lemma}

In the next two examples, we will consider the rolling procedure mixed by random slipping. To indicate the slipping moments and slipping duration, we need a Poisson point process. For this purpose, we introduce, for each $\e>0$, a subordinator $S^\e = \{S^\e_t\}_{t\ge0}$, which has zero drift and jump measure $\nu^\e$ with $\nu^\e((0,\infty))<\infty$. Note that in this case, each $S^\e$ is also a one-dimensional compound Poisson process, with rate $\lambda(\e) := \nu^\e((0,\infty))$ and jump size distribution $\mu^\e:= \frac{1}{\lambda(\e)} \nu^\e$.
In addition, almost surely, the jumping times of each $S^\e$ are infinitely many and countable in increasing order (see \cite[Theorem 21.3]{Sat99}). We define the jumping time of each $S^\e$ recursively by
\begin{equation}\label{jump-time}
  \begin{split}
    \tau^\e_1 &= \inf\{t\ge0: \Delta S^\e_t \ne 0\}, \\
  \tau^\e_{k+1} &= \inf\{t>\tau^\e_k: \Delta S^\e_t \ne 0\}, \quad k\in \N_+.
  \end{split}
\end{equation}
These stopping times indicate the moments when every random slipping will occur. Let $e^\e = \{e^\e(t)\}_{t\ge0}$ be the associated Poisson point process valued on $[0,\infty)$, namely, $e^\e = \Delta S^\e$. The value of $e^\e$ on each jumping time $\tau^\e_k$ indicates the duration of each random slipping. We refer to \cite{Ber96,Sat99} for more on the topics of subordinators, compound Poisson processes and Poisson point processes.
\begin{figure}[htbp]
  \centering

  \includegraphics[width=0.8\textwidth]{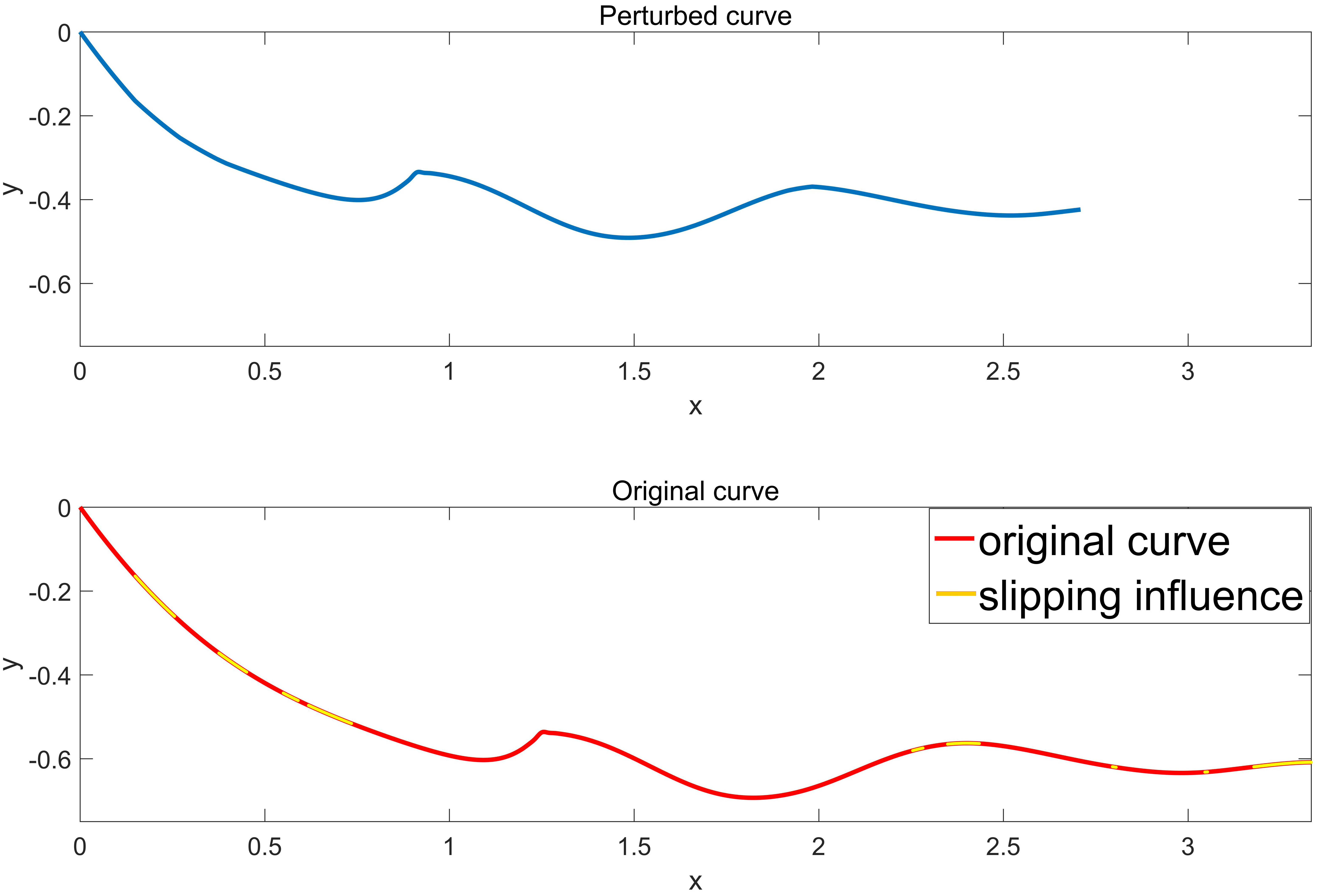}
  \caption{``Slipping mode'' in Example \ref{EX1}}
  \label{53C}
\end{figure}

\begin{example}[Translational slipping]\label{EX1}
  Consider the case that each slipping is translational along the given curve, that is, the contact point on the manifold will be rest while the contact point on  Euclidean space will move along $\gamma$ from the slipping starting point to the slipping end point. As is shown in Fig.~\ref{53C}, slipping will cut off the yellow parts in the original curve, and connect the remaining parts piece-by-piece. The perturbed curve is marked blue in Fig.~\ref{53C}. To be precise, the equivalent Euclidean curve that one can slip the manifold along is given by
  \begin{equation*}
    \gamma^\e_t =
    \begin{cases}
      \gamma_t, & 0\le t< \tau^\e_1, \\
      \gamma_{\tau^\e_1}, & \tau^\e_1 \le t< \tau^\e_1 + e^\e(\tau^\e_1), \\
      \gamma_t - \gamma_{\tau^\e_1+e^\e(\tau^\e_1)} + \gamma_{\tau^\e_1}, & \tau^\e_1 + e^\e(\tau^\e_1) \le t < \tau^\e_1 + e^\e(\tau^\e_1) + \tau^\e_2, \\
      \cdots, &
    \end{cases}
  \end{equation*}
  for each $\e>0$. Then each $\gamma^\e$ is locally of finite variation, and we have
  \begin{equation*}
    \sup_{0\le t\le T}|\gamma^\e_t-\gamma_t| \le C \sum_{0\le t\le T} e^\e(t) = C S^\e_T.
  \end{equation*}
  If the family $\{S^\e\}_{\e>0}$ are supposed to satisfy that
  \begin{equation}\label{subord}
    \limsup_{\e\to0} \e \log \P^\e\left( S^\e_T \ge\eta \right) = -\infty, \quad \text{for each } T>0 \text{ and } \eta>0,
  \end{equation}
  then the LDP for $\{\gamma^\e\}_{\e>0}$ holds, by Lemma \ref{LDP-char}, with the good rate function
  \begin{equation}\label{LDP-exmp}
    I^\flat(y) =
    \begin{cases}
      \int_0^\infty \sup_{\lambda\in\R^d}\langle\lambda, \dot y(t)- \dot\gamma(t)\rangle dt, & y \text{ is absolutely continuous}, \\
      \infty, & \text{otherwise}.
    \end{cases}
  \end{equation}
  Moreover, it is easy to see that $|V(\gamma^\e)| \le |V(\gamma)|$, and the exponential tightness of $\{|V(\gamma^\e)|_t\}$, for each $t>0$, follows. Therefore, Theorem \ref{non-cpct-1} is applicable here.

\begin{figure}[htbp]
  \centering
  \subfloat[ ]{
      \begin{minipage}{195pt}
          \centering
          \label{53CP}
          \includegraphics[width=1\textwidth]{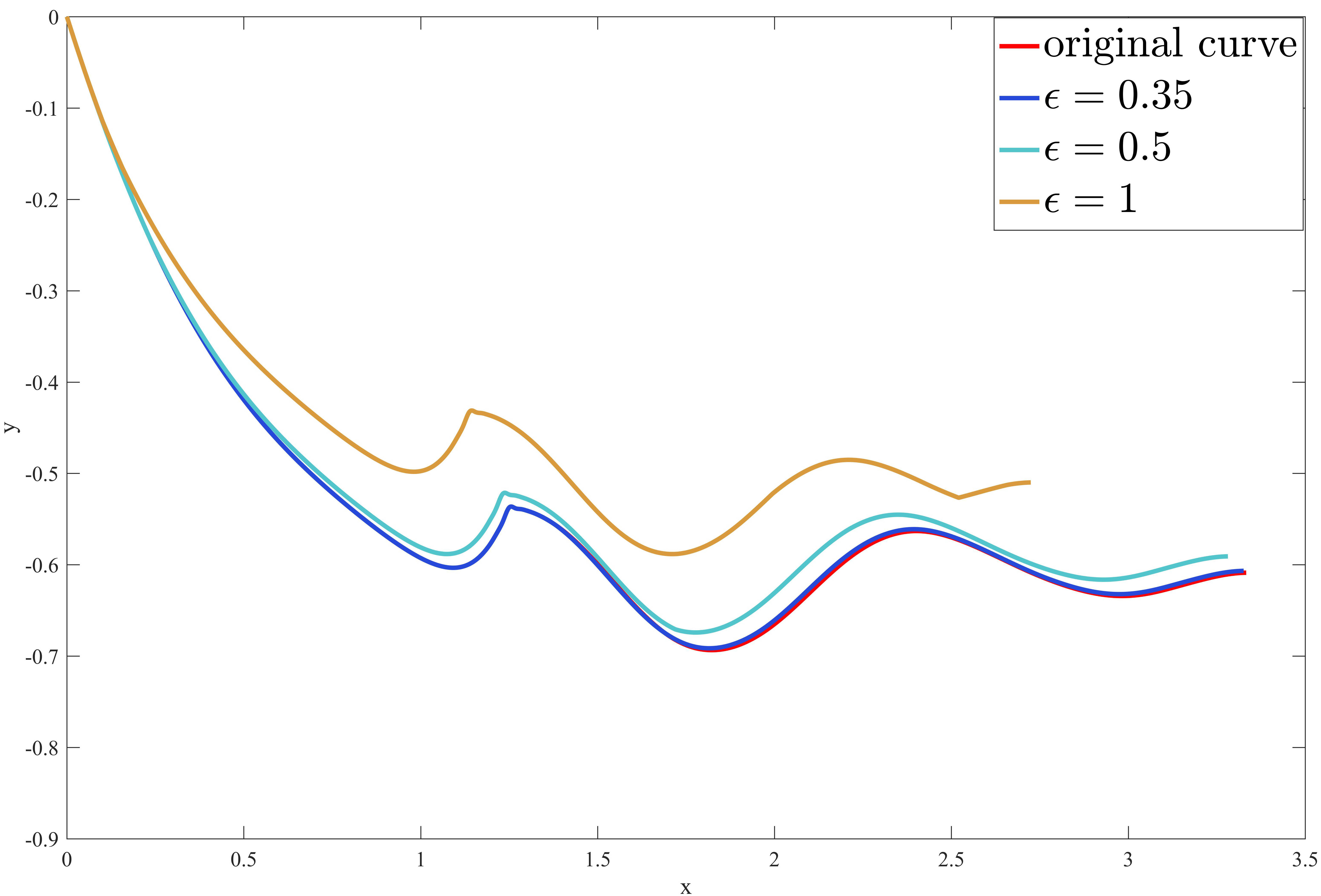}
      \end{minipage}
  }
  \subfloat[ ]{
      \begin{minipage}{145pt}
          \centering
          \label{53CS}
          \includegraphics[width=1\textwidth]{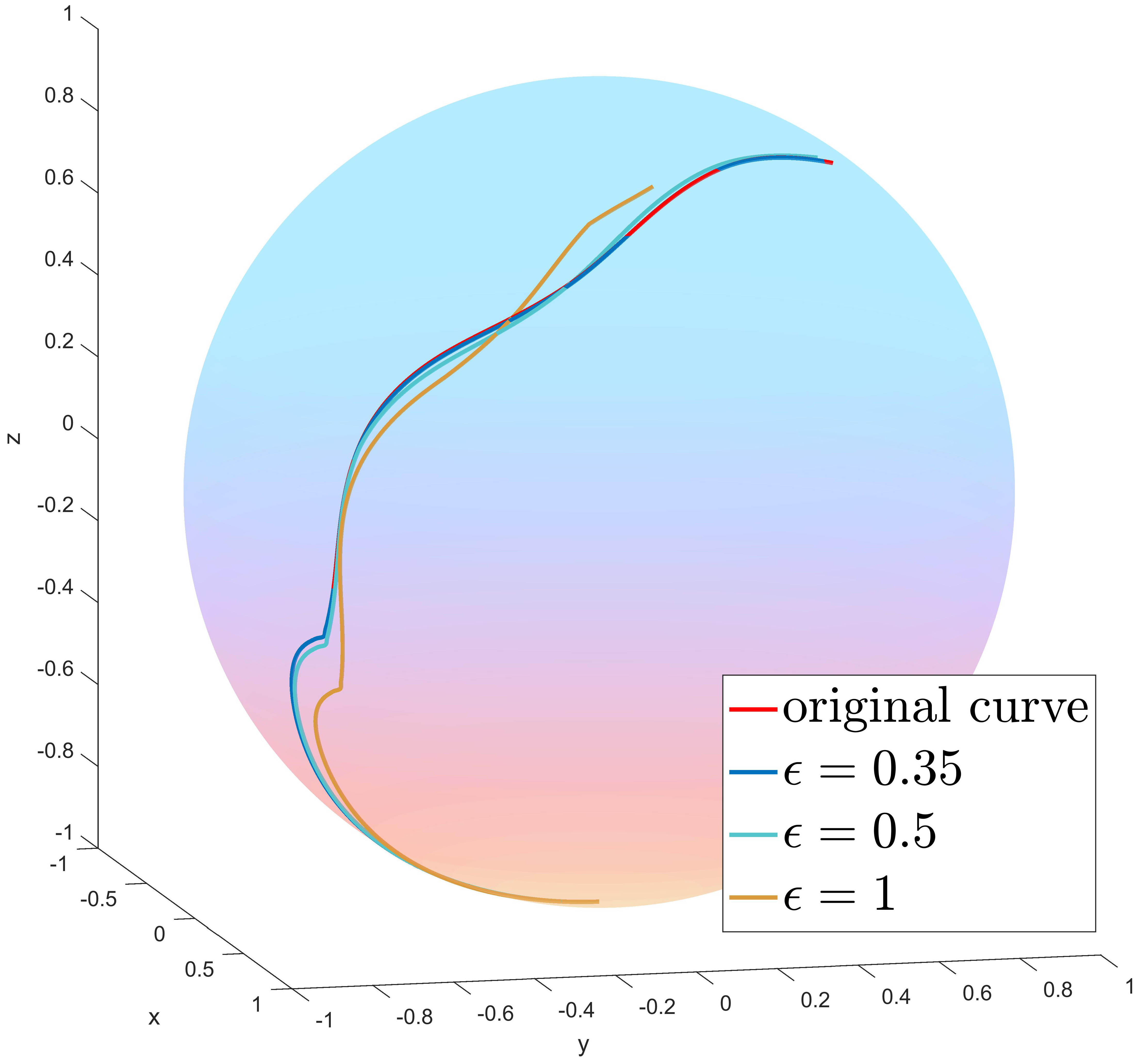}
      \end{minipage}
  }
  \caption{The original and perturbed curves on the plane and their traces on the unit ball, in Example \ref{EX1}} \label{53}
\end{figure}

  A sufficient condition for the assumption \eqref{subord} in terms of the jump measure $\nu^\e$ can be found. Indeed, by Chebyshev's inequality we have
  $\P^\epsilon(S_T^\epsilon \geq \eta) \leq \frac{1}{\eta}\E^\epsilon S_T^\epsilon$.
  If we force
  $$
    \lim_{\epsilon \to 0} \epsilon \log \E^\epsilon \left(S_T^\epsilon\right) = -\infty,
  $$
  then \eqref{subord} is fulfilled. Using the relation between the mean and the characteristic function, we have
  $$
    \E^\epsilon \left(S_T^\epsilon\right) = T \int_0^\infty x \nu^\epsilon(dx).
  $$
  Thus we can propose a sufficient condition \eqref{subord} as follows,
  \begin{equation}\label{suff-cond}
    \lim_{\epsilon \to 0} \epsilon \log \int_0^\infty x \nu^\epsilon(dx)=-\infty.
  \end{equation}
  The condition \eqref{suff-cond} means asymptotically as $\epsilon \to 0$,
  $$
    \int_0^\infty x \nu^\epsilon(dx) \sim \exp\left(\frac{-l(\epsilon)}{\epsilon }\right),
  $$
  where $l(\epsilon) \to +\infty$ as $\epsilon \to 0$. A suitable examples for  $\nu^\epsilon$ is
  $$
    \nu^\epsilon(dx)=f\left(x\exp\left(\frac{l(\epsilon)}{2\epsilon }\right)\right) dx,
  $$
  where $f$ is some function on $\Ro d$ satisfying $\int_0^\infty xf(x)dx<\infty$.

  As an example of visualization, we take $f(x) = e^{-x}$ and $l(\e) = 2\e^{-0.1}$ so that
  \begin{equation}\label{exmp-nu}
    \nu^\epsilon(dx)=\exp\left[-x\exp\left(1/\epsilon^{1.1}\right)\right]dx.
  \end{equation}
  We visualize the original and perturbed curves both on the plane and on the unit ball in Fig.~\ref{53}. The perturbed curves on the unit ball are just the traces left on the ball when being rolled along the original planar curve with random slipping. It shows that as $\epsilon$ getting smaller, the perturbed curves gets closer to the original one both on the plane and on the ball.
  \qed
\end{example}

\begin{figure}[htbp]
  \centering
  \includegraphics[width=0.8\textwidth]{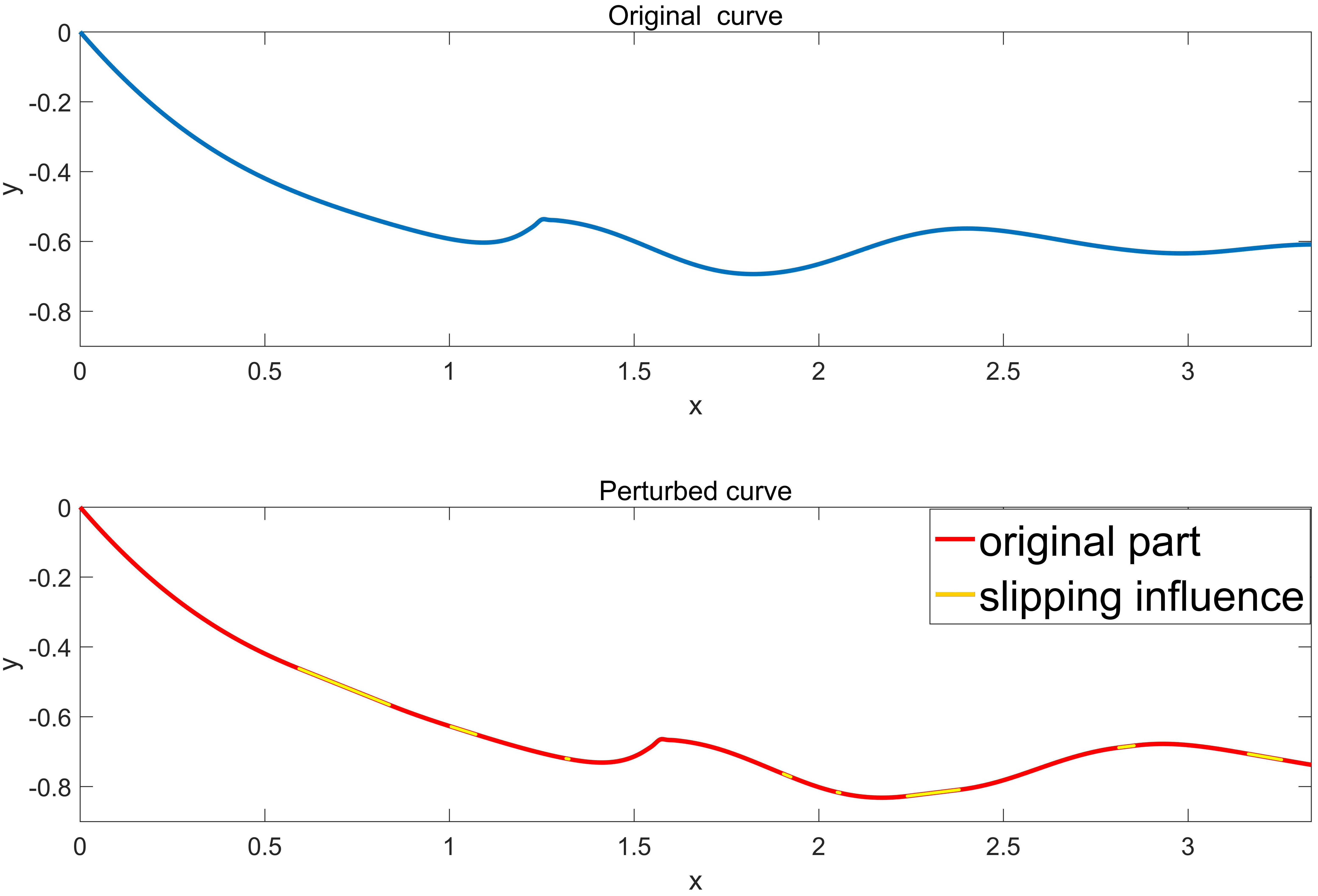}
  \caption{``Slipping mode'' in Example \ref{EX2}}
  \label{54C}
\end{figure}

\begin{example}[Slipping in place]\label{EX2}
  There is another slipping mode: the slipping only happens in-place. In this mode, the contact point on the Euclidean space will rest, but the contact point on the manifold will move along the geodesic that starts from the slipping starting point with initial speed the tangent vector of the curve $\gamma$ at the slipping starting point. As is shown in in Fig.~\ref{54C}, slipping will extend the original curve with straight lines along the tangent direction of the slipping starting points. These straight lines are marked yellow in the perturbed curve. For each $\e>0$, the equivalent curve is given by
  \begin{equation*}
    \gamma^\e_t =
    \begin{cases}
      \gamma_t, & 0\le t< \tau^\e_1, \\
      \gamma_{\tau^\e_1} +(t-\tau^\e_1)\dot\gamma_{\tau^\e_1}, & \tau^\e_1 \le t< \tau^\e_1 + e^\e(\tau^\e_1), \\
      \gamma_{t- e^\e(\tau^\e_1)} + e^\e(\tau^\e_1) \dot\gamma_{\tau^\e_1}, & \tau^\e_1 + e^\e(\tau^\e_1) \le t < \tau_1 + e^\e(\tau^\e_1) + \tau^\e_2, \\
      \cdots. &
    \end{cases}
  \end{equation*}
  Then we have
  \begin{equation*}
    \sup_{0\le t\le T}|\gamma^\e_t-\gamma_t| \le 2C \sum_{0\le t\le T} e^\e(t) = 2C S^\e_T.
  \end{equation*}
  Same as before, the condition \eqref{subord} yields the LDP for $\{\gamma^\e\}_{\e>0}$, with the good rate function \eqref{LDP-exmp}. On the other hand, it is easy to derive
  \begin{equation*}
    |V(\gamma^\e)|_t \le |V(\gamma)|_t + C \sum_{0\le s\le t} e^\e(s) \le C(t+S_t^\e).
  \end{equation*}
  Hence, condition \eqref{subord} also implies the exponential tightness of $\{|V(\gamma^\e)|_t\}$, for each $t>0$.

  The original and perturbed curves on the plane and the unit ball in this slipping mode are visualized in Fig.~\ref{54}, with the same $\nu^\e$ as the one in \eqref{exmp-nu}.
  \qed
  \begin{figure}
    \centering
    \subfloat[ ]{
        \begin{minipage}{195pt}
            \centering
            \label{54CP}
            \includegraphics[width=1\textwidth]{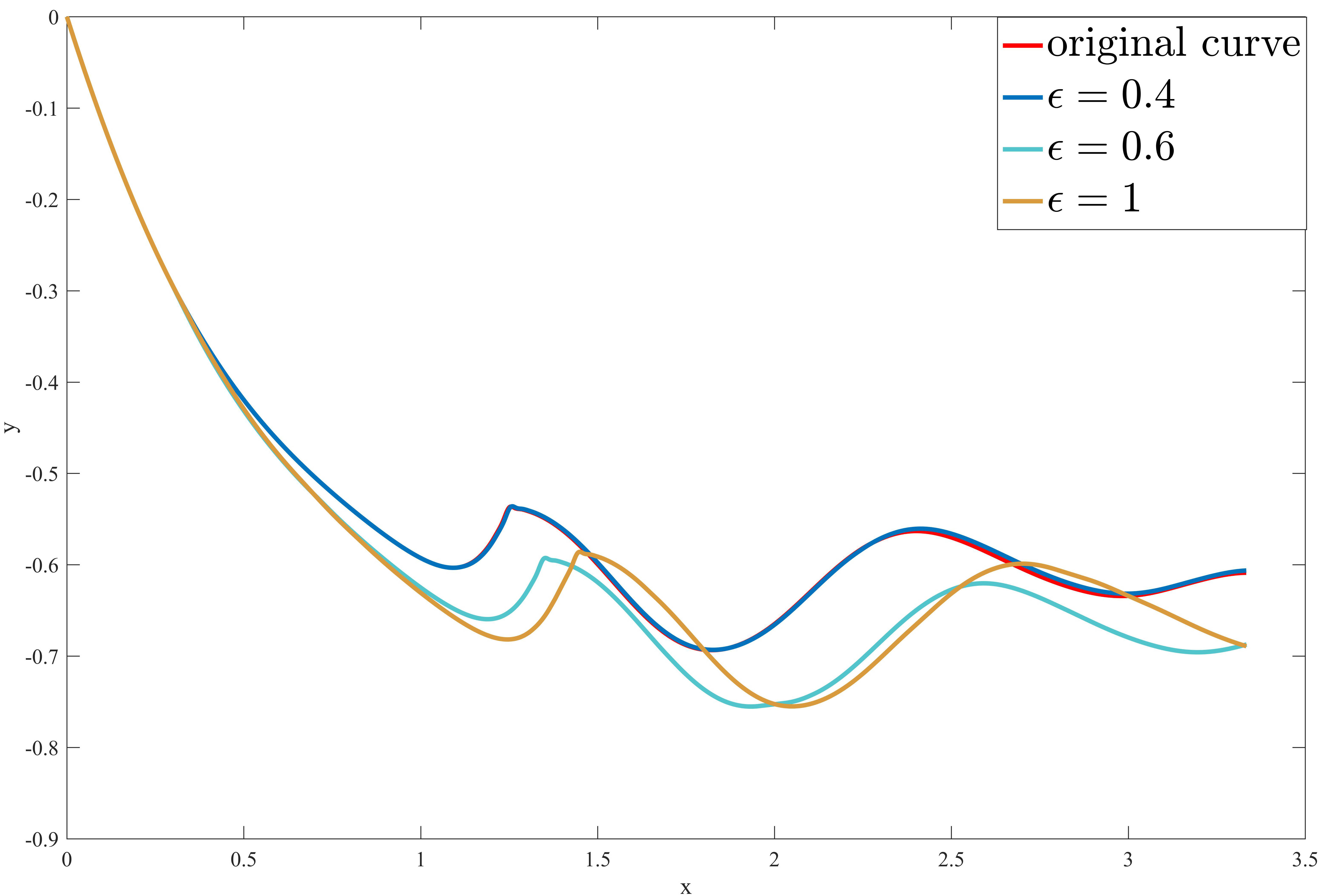}
        \end{minipage}
    }
    \subfloat[ ]{
        \begin{minipage}{145pt}
            \centering
            \label{54CS}
            \includegraphics[width=1\textwidth]{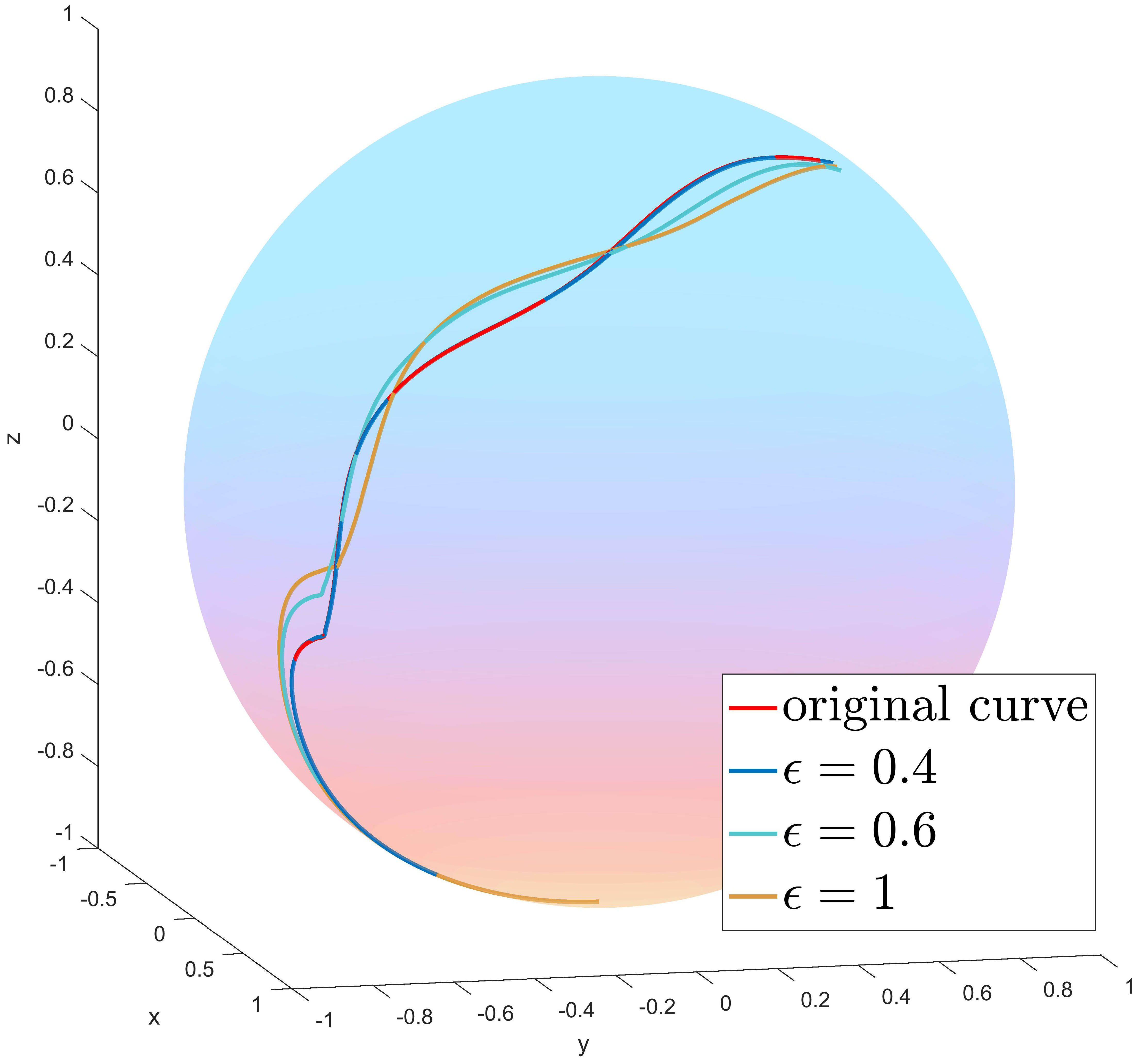}
        \end{minipage}
    }
    \caption{The original and perturbed curves on the plane and their traces on the unit ball, in Example \ref{EX2}} \label{54}
  \end{figure}
\end{example}

%

\begin{example}[Piecewise linear approximation]\label{EX3}

  In this example, we consider the rolling along the piecewise linear approximation of $\gamma$. Assume in addition that the speed function $\dot\gamma$ is Lipschitz with Lipschitz constant less than $C>0$. For each $\e>0$, let $\{\tau^\e_k\}_{k=1}^\infty$ be the sequence of stopping times defined in \eqref{jump-time}. Set $\tau_0^\e\equiv0$ for simplicity. We define the speed function of the approximation curve by
  \begin{equation*}
    \xi^\e_t = \dot\gamma_{\tau^\e_k}, \quad\text{when } \tau^\e_k \le t < \tau^\e_{k+1}, k\ge0.
  \end{equation*}
  Then $|\xi_t^\e|<C$ for $t\ge0$ and $\e>0$. Define the approximation curves by
  \begin{equation*}
    \gamma^\e_t = \int_0^t \xi^\e_s ds, \quad \e>0.
  \end{equation*}
  Then each $\gamma^\e$ is piecewise linear and locally of finite variation with variation satisfying $|V(\gamma^\e_t)| \le Ct$. Hence, for each $t>0$, the family $\{|V(\gamma^\e)|_t \}_{\e>0}$ is exponentially tight.

  Since $\dot\gamma$ is Lipschitz, we have
  \begin{equation*}
    \sup_{0\le t\le T}|\gamma^\e_t-\gamma_t| \le \int_0^T |\xi^\e_s-\dot\gamma_s| ds \le
    \begin{cases}
      CT^2, & 0<T\le\tau^\e_1, \\
      CT \sup_{0\le k\le m-1} (\tau^\e_{k+1} - \tau^\e_k) \vee (T-\tau_m^\e), & \tau^\e_m < T \le \tau^\e_{m+1}, m\ge1.
    \end{cases}
  \end{equation*}
  Then for every $T>0$ and any $\eta>0$,
  \begin{equation}\label{est-13}
    \begin{split}
      &\ \P^\e \left( \sup_{0\le t\le T} |\gamma^\e_t - \gamma_t| \ge \eta \right) \\
      \le&\ \P^\e\left( CT^2\ge \eta, \tau^\e_1\ge T \right) + \sum_{m=1}^\infty \P^\e \left( CT \sup_{0\le k\le m-1} (\tau^\e_{k+1} - \tau^\e_k) \vee (T-\tau_m^\e) \ge \eta, \tau^\e_m < T \le \tau^\e_{m+1} \right)  \\
      =: &\ I^\e + \sum_{m=1}^\infty J^\e_m.
    \end{split}
  \end{equation}
  Denote $\lambda(\e) := \nu^\e((0,\infty))$. Since the sequence $\{\tau^\e_{k+1} - \tau^\e_k\}_{k=0}^\infty$ constitutes independent identically distributed random variables, each exponentially distributed with mean $1/\lambda(\e)$ (see the proof of \cite[Theorem 21.3]{Sat99}), we have
  \begin{equation}\label{est-14}
    I^\e = \ind_{\{CT^2\ge \eta\}} \P^\e\left( \tau^\e_1\ge T \right) = \ind_{\{CT^2\ge \eta\}} e^{-\lambda(\e) T},
  \end{equation}
  and
  \begin{equation}\label{est-15}
    \begin{split}
      J^\e_m = &\ \P^\e \left( CT \sup_{0\le k\le m-1} (\tau^\e_{k+1} - \tau^\e_k) < \eta \right) \P^\e \left( CT(T-\tau_m^\e) \ge \eta, \tau^\e_m < T \le \tau^\e_{m+1} \right) \\
         &\ +\P^\e \left( CT \sup_{0\le k\le m-1} (\tau^\e_{k+1} - \tau^\e_k) \ge \eta \right) \P^\e \left( \tau^\e_m < T \le \tau^\e_{m+1} \right).
    \end{split}
  \end{equation}
  Using the fact that $\tau^\e_m$ obeys the Gamma distribution $\text{Gamma}(m,\lambda(\e))$, we deduce
  \begin{equation}\label{est-16}
    \begin{split}
      &\ \P^\e \left( \tau^\e_m < T \le \tau^\e_{m+1} \right) = \E^\e \left[ \P^\e \left( \tau^\e_m < T \le \tau^\e_{m+1} \mid \tau^\e_{m+1} - \tau^\e_m \right) \right] \\
      =&\ \int_0^T \P^\e \left( \tau^\e_m < T \le \tau^\e_m +t \right) \lambda(\e) e^{-\lambda(\e) t} dt + \int_T^\infty \P^\e \left( \tau^\e_m < T \right) \lambda(\e) e^{-\lambda(\e) t} dt \\
      =&\ \left( \int_0^T \int_{T-t}^T  + \int_T^\infty \int_0^T \right) \frac{\lambda(\e)^m x^{m-1} e^{-\lambda(\e)x}}{(m-1)!} \lambda(\e) e^{-\lambda(\e) t} dxdt \\
      =&\ e^{-\lambda(\e)T}\frac{\lambda(\e)^m T^m}{m!},
    \end{split}
  \end{equation}
  and similarly,
  \begin{equation}\label{est-17}
    \begin{split}
      &\ \P^\e \left( CT(T-\tau_m^\e) \ge \eta, \tau^\e_m < T \le \tau^\e_{m+1} \right) \\
      =&\ \ind_{\{CT^2\ge \eta\}} \P^\e \left( \tau^\e_m + \eta/(CT) < T \le \tau^\e_{m+1} \right) \\
      =&\ \ind_{\{CT^2\ge \eta\}}e^{-\lambda(\e)T} \frac{\lambda(\e)^m (T-\frac{\eta}{CT})^m}{m!}.
    \end{split}
  \end{equation}
  Using again the fact that the sequence $\{\tau^\e_{k+1} - \tau^\e_k\}_{k=0}^\infty$ is i.i.d. with exponential distribution,
  \begin{equation}\label{est-18}
    \P^\e \left( CT \sup_{0\le k\le m-1} (\tau^\e_{k+1} - \tau^\e_k) < \eta \right) = \prod_{k=0}^{m-1} \P^\e \left( CT (\tau^\e_{k+1} - \tau^\e_k) < \eta \right) = \left( 1- e^{-\lambda(\e)\frac{\eta}{CT}}\right)^m.
  \end{equation}
  Combining \eqref{est-13}--\eqref{est-18}, we get
  \begin{equation}\label{est-12}
    \begin{split}
      \P^\e \left( \sup_{0\le t\le T} |\gamma^\e_t - \gamma_t| \ge \eta \right) \le&\ \ind_{\{CT^2\ge \eta\}} e^{-\lambda(\e) T} \sum_{m=0}^\infty \frac{\lambda(\e)^m (T-\frac{\eta}{CT})^m}{m!} \left( 1- e^{-\lambda(\e)\frac{\eta}{CT}}\right)^m \\
         &\ + e^{-\lambda(\e)T} \sum_{m=1}^\infty \frac{\lambda(\e)^m T^m}{m!} \left[ 1-\left( 1- e^{-\lambda(\e)\frac{\eta}{CT}}\right)^m \right] \\
         =&\ \ind_{\{CT^2\ge \eta\}} \exp\left\{ -\lambda(\e) \left( \frac{\eta}{CT} + Te^{-\lambda(\e)\frac{\eta}{CT}} - \frac{\eta}{CT} e^{-\lambda(\e)\frac{\eta}{CT}} \right)\right\} \\
         &\ + \left(1- e^{-\lambda(\e)T}- \exp\left\{ -\lambda(\e)T e^{-\lambda(\e)\frac{\eta}{CT}} \right\} \right) \\
         \le&\ \exp\left\{ -\lambda(\e) \left( \frac{\eta}{CT} + Te^{-\lambda(\e)\frac{\eta}{CT}} - \frac{\eta}{CT} e^{-\lambda(\e)\frac{\eta}{CT}} \right)\right\} \\
         &\ + \left(1- \exp\left\{ -\lambda(\e)T e^{-\lambda(\e)\frac{\eta}{CT}} \right\} \right).
    \end{split}
  \end{equation}

\begin{figure}
    \centering
    \subfloat[ ]{
    \begin{minipage}{195pt}
        \centering
        \label{55CP}
        \includegraphics[width=1\textwidth]{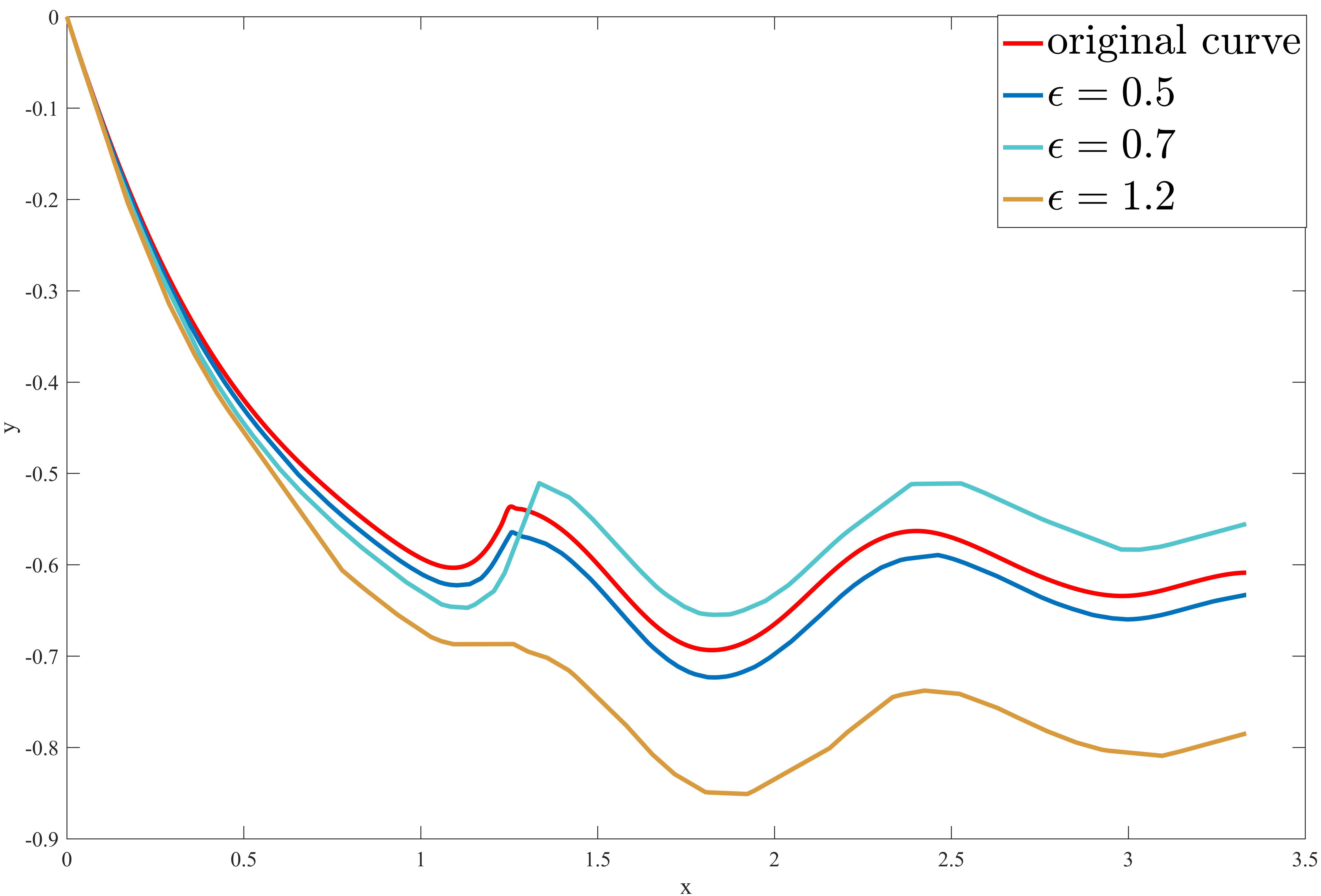}
    \end{minipage}
    }
    \subfloat[ ]{
    \begin{minipage}{145pt}
        \centering
        \label{55CS}
        \includegraphics[width=1\textwidth]{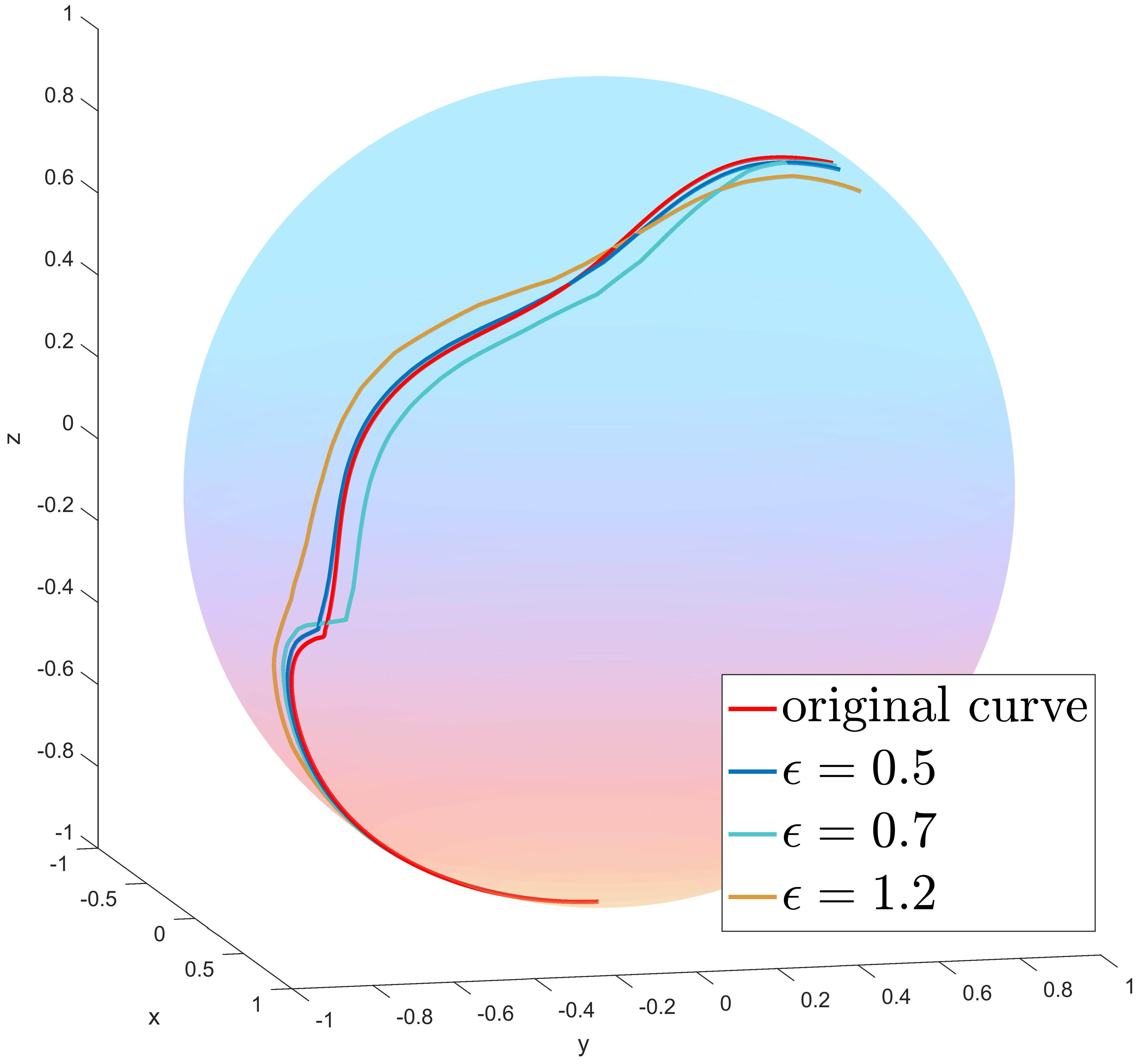}
    \end{minipage}
    }
\caption{The original and perturbed curves on the plane and their traces on the unit ball, in Example \ref{EX3}.} \label{55}
\end{figure}
  Now we assume that $\lim_{\e\to0} \e\lambda(\e) = \infty$, that is,
  \begin{equation}\label{assump}
    \lim_{\e\to0} \e\nu^\e((0,\infty)) = \infty.
  \end{equation}
  In particular, $\lim_{\e\to0} \lambda(\e) = \infty$. Then we take $\e\log$ on both sides of \eqref{est-12}. It is obvious that the $\e\log$ of the first term behind the last inequality of \eqref{est-12} goes to $-\infty$ as $\e\to0$. For the second term, we use Taylor's theorem to obtain that there exists a $(0,1)$-valued function $\theta(\e)$ such that
  \begin{equation*}
    \begin{split}
      \e\log \left( 1- \exp\left\{ -\lambda(\e)T e^{-\lambda(\e)\frac{\eta}{CT}} \right\} \right) & = \e\log \left( \lambda(\e)T e^{-\lambda(\e)\frac{\eta}{CT}} \exp\left\{ -\theta(\e)\lambda(\e)T e^{-\lambda(\e)\frac{\eta}{CT}} \right\} \right) \\
         & = \e\log \lambda(\e) + \e\log T  - \e\lambda(\e)\frac{\eta}{CT} -\e\theta(\e)\lambda(\e)T e^{-\lambda(\e)\frac{\eta}{CT}},
    \end{split}
  \end{equation*}
  which goes to $-\infty$ as $\e\to0$. Therefore,
  \begin{equation*}
    \limsup_{\e\to0} \e\log \P^\e\left( \sup_{0\le t\le T} |\gamma^\e_t - \gamma_t| \ge \eta \right) = -\infty.
  \end{equation*}
  By Lemma \ref{LDP-char}, the family $\{\gamma^\e\}_{\e>0}$ satisfies the LDP with the good rate function \eqref{LDP-exmp}.

  As a contrast with \eqref{exmp-nu}, we now take
  $$\nu^\epsilon(dx)=\exp\left[-x\exp\left(-1/\epsilon^{1.1}\right)\right]dx,$$
  which can be easily verified to satisfy the condition \eqref{assump}. In this case, the original and the piecewise linearly approximated curves on the plane and the unit ball are visualized in Fig.~\ref{55}.
  \qed

\end{example}

\section{Appendix: Proofs of Proposition \ref{LDP-control} and Corollary \ref{local-Lip}}\label{Appendix}

\renewcommand{\theequation}{A.\arabic{equation}}
\renewcommand{\thefigure}{A.\arabic{figure}}
\setcounter{equation}{0}
\setcounter{figure}{0}
\renewcommand{\thetheorem}{A.\arabic{theorem}}

This appendix provides the proofs for LDP of the controlled SDE \eqref{SDE-Ito} (Proposition \ref{LDP-control} and Corollary \ref{local-Lip}).

\subsection{Lemmas}

The following lemma is a small modification of \cite[Theorem 1.2, Lemma 2.5]{Gar08}, which will be useful in proving the LDP results.
\begin{lemma}\label{LDP-integral}
  Suppose the family $\{G(Y^\e)_t\}_{\e>0}$ is exponentially tight for each $t>0$.
  Let $\{F^\e\}_{\e>0}$ be a family of continuous $\{\F^\e_t\}_{t\ge0}$-adapted $\R^d\otimes (\R^n)^*$-valued processes. If the family $\{(Y^\e,U^\e,F^\e)\}_{\e>0}$ satisfies the LDP with a good rate function $I^\sharp$, then the family $\{(F^\e\cdot Y^\e, Y^\e,U^\e,F^\e)\}_{\e>0}$ also satisfies the LDP with the following good rate function:
  \begin{equation*}
    I(w,y,u,f) =
    \begin{cases}
      I^\sharp(y,u,f), & w = f\cdot y \text{ and } y \text{ is locally of finite variation}, \\
      \infty, & \text{otherwise}.
    \end{cases}
  \end{equation*}
\end{lemma}
\begin{proof}
   Firstly, we assert from \cite[Lemma 2.5]{Gar08} that the family $\{Y^\e\}_{\e>0}$ satisfies the so called ``uniform exponential tightness'' condition proposed in \cite[Definition 1.1]{Gar08}. Next, we observe that $F^\e\cdot Y^\e = (U^\e,F^\e) \cdot (0,Y^\e)^T$. The result is a corollary of \cite[Theorem 1.2]{Gar08}.
\end{proof}

The following Gronwall-type inequality is adapted from \cite[Lemma IX.6.3]{JS13}. The proof is almost the same and shall be omitted.
\begin{lemma}\label{Gronwall}
  Let $A$ be a nondecreasing continuous process, $H$ be a nonnegative continuous process, such that $\E( (H\cdot A)_\infty ) < \infty$ and $A_\infty \le K$ identically for some constant $K$.
  Suppose that for each stopping time $T$ we have
  \begin{equation*}
    \E(H_T) \le \alpha + \E( (H\cdot A)_T ),
  \end{equation*}
  for some constant $\alpha$. Then $\E(H_\infty) \le \alpha e^{Kt}$.
\end{lemma}

\begin{lemma}\label{est-8}
  Suppose the family of increasing processes $\{G(Y^\e)\}_{\e>0}$ is uniformly bounded by a constant $K>0$. Let $F^\e$ be an $\{\F^\e_t\}$-adapted $\R^d\otimes (\R^n)^*$-valued processes, $z_0$ be a constant vector in $\R^d$, and let
  \begin{equation*}
    dZ^\e_t = F^\e_{t} dY^\e_t, \quad Z^\e_0 = z_0.
  \end{equation*}
  Let $T^\e$ be an $\{\F^\e_t\}$-stopping time. Suppose there exist positive constants $C$ and $\rho$, such that for any $t\in [0,T^\e]$,
  \begin{equation}\label{bound}
    |F^\e_t| \le C(\rho^2 + |Z^\e_t|^2)^{1/2}.
  \end{equation}
  Then for any $a>0$ and $0<\e \le 1$,
  \begin{equation*}
    \e \log \P^\e\left( \sup_{t\in[0,T^\e]} |Z^\e_t|\ge a \right) \le K'+ \log\left( \frac{\rho^2 + |z_0^2|}{\rho^2+a^2}\right),
  \end{equation*}
  where $K' = \left(2C+(2+n)C^2\right) K$.
\end{lemma}
\begin{proof}
  Define $\phi^\e(z) = (\rho^2 + |z|^2)^{1/\e}$ for each $\e>0$. Then
  \begin{equation}\label{est-7}
    \partial_i \phi^\e(z) = \frac{2\phi^\e(z)}{\e(\rho^2+|z|^2)}z_i, \quad \partial_i\partial_j \phi^\e(z) = \frac{2\phi^\e(z)}{\e(\rho^2+|z|^2)} \left( \delta_{ij} + 2\left( \frac{1}{\e}-1 \right) \frac{z_i z_j}{\rho^2+|z|^2}\right).
  \end{equation}
  Let $\Phi^\e_t := \phi^\e(Z^\e_t)$. By It\^o's formula (see, e.g., \cite[Theorem I.4.57]{JS13}),
  \begin{equation*}
    \Phi^\e = \phi^\e(z_0) + \partial_i\phi^\e(Z^\e)F_j^{\e,i} \cdot A^{\e,j} + \partial_i\phi^\e(Z^\e)F^{\e,i}_{j} \cdot M^{\e,j} + \frac{1}{2} \partial_i\partial_j\phi^\e(Z^\e) F_k^{\e,i} F_h^{\e,j} \cdot \langle M^{\e,k}, M^{\e,h} \rangle
  \end{equation*}
  Define a stopping time $T^{\e,a} = \inf\{t\ge0: |Z^\e_t|\ge a\} \wedge T^\e$. Using the bound \eqref{bound}, for each $\e>0$,
  \begin{equation*}
    \E^\e \left( \int_0^{T^{\e,a}} |D\phi^\e(Z^\e)F^\e|^2 d|\langle M^\e, M^\e \rangle|_t\right) < \infty.
  \end{equation*}
  Then the stochastic integrals $\partial_i\phi^\e(Z^\e)F^{\e,i}_{j} \cdot M^{\e,j}$ is martingale up to $T^{\e,a}$. The optional sampling theorem (see, e.g., \cite[Theorem 1.3.22]{KS91}) yields for any $\{\F^\e_t\}$-stopping time $S^\e$ that
  \begin{equation*}
    \begin{split}
       \E^\e\left( \Phi^\e_{S^\e \wedge T^{\e,a}} \right) =&\ \phi^\e(z_0) + \E^\e\left( \int_0^{S^\e \wedge T^{\e,a}} \partial_i\phi^\e(Z^\e_s)F_{j,s}^{\e,i} dA^{\e,j}_s \right) \\
       &\ + \frac{1}{2} \E^\e\left( \int_0^{S^\e \wedge T^{\e,a}} \partial_i\partial_j\phi^\e(Z^\e_s) F_{k,s}^{\e,i} F_{h,s}^{\e,j} d \langle M^{\e,k}, M^{\e,h} \rangle_s \right) \\
      =:&\ \phi^\e(z_0) + I^\e_1 +I^\e_2.
    \end{split}
  \end{equation*}
  By \eqref{bound} and \eqref{est-7}, it is easy to get
  \begin{equation*}
    |I^\e_1| \le \frac{2C}{\e} \E^\e \int_0^{S^\e \wedge T^{\e,a}} \Phi^\e_s d|V(A^\e)|_s, \quad |I^\e_2| \le \frac{(2+n)C^2}{\e^2} \E^\e \int_0^{S^\e \wedge T^{\e,a}} \Phi^\e_s d|\langle M^\e, M^\e \rangle|_s.
  \end{equation*}
  Then
  \begin{equation*}
    \E^\e\left( \Phi^\e_{S^\e \wedge T^{\e,a}} \right) \le \phi^\e(z_0) + \left(2C+(2+n)C^2\right) \frac{1}{\e} \E \int_0^{S^\e} \Phi^\e_{s \wedge T^{\e,a}} dG^\e_s.
  \end{equation*}
  By virtue of the Gronwall-type inequality in Lemma \ref{Gronwall}, we have
  \begin{equation*}
    \E^\e(\Phi^\e_{T^{\e,a}}) \le \phi^\e(z_0)e^{K'/\e}.
  \end{equation*}
  Then by the definition of $T^{\e,a}$ and Chebyshev's inequality,
  \begin{equation*}
    \begin{split}
       \e \log \P^\e\left(\sup_{0\le t\le T^\e} |Z^\e_t|\ge a \right) & = \e \log \P^\e(|Z^\e_{T^{\e,a}}| \ge a) = \e \log \P^\e(\Phi^\e_{T^{\e,a}} \ge \phi(a)) \\
         & \le K' + \e \log \phi^\e(z_0) - \e \log \phi^\e(a) = K' +\log\left( \frac{\rho^2 + |z_0^2|}{\rho^2+a^2}\right).
    \end{split}
  \end{equation*}
  The proof is completed.
\end{proof}

We introduce the supremum norm and Lipschitz norm of a function $F$ on $\R^d\times\R^l$ by
\begin{align*}
  \|F\|_0 &:= \sup_{(x,u)\in\R^d\times\R^l} |F(x,u)|, \\
  \|F\|_\Lip &:= \|F\|_0 + \sup_{\begin{subarray}{c} (x_1,u_1),(x_2,u_2)\in\R^d\times\R^l \\ x_1\ne x_2,u_1\ne u_2 \end{subarray}} \frac{|F(x_1,u_1) - F(x_2,u_2)|}{|x_1-x_2|+|u_1-u_2|},
\end{align*}
where the notation of absolute value means the matrix norm or vector norm.

For each $t>0$ and $m\in\N_+$, we denote $[t]_m:=\lfloor t\rfloor + \frac{\lfloor m(t-\lfloor t\rfloor)\rfloor}{m}$. Obviously we have $|t-[t]_m| < \frac{1}{m}$. For each $m\in\N_+$, define $X^{\e,m}$ to be the solution of the following SDE:
\begin{equation}\label{SDE-app}
  dX^{\e,m}_t = F(X^{\e,m}_{[t]_m}, U^{\e}_{[t]_m}) dY^\e_t, \quad X^{\e,m}_0 = x_0.
\end{equation}
The following lemma shows that $X^{\e,m}$ are exponentially good approximations of $X^\e$.
\begin{lemma}\label{exp-good-app}
  Suppose that the family of increasing processes $\{G(Y^\e)\}_{\e>0}$ is uniformly bounded by a constant $K>0$, and that the family $\{(Y^\e,U^\e)\}_{\e>0}$ is  exponentially tight. Assume further that the function $F$ to be bounded and global Lipschitz. Then for each $T>0$ and $a>0$,
  \begin{equation*}
    \lim_{m\to\infty} \limsup_{\e\to0} \e\log \P^\e(\|X^{\e,m} - X^\e\|_T >a) = -\infty.
  \end{equation*}
\end{lemma}
\begin{proof}
  Fix $m\in\N_+$. Let $Z^\e := X^{\e,m} - X^\e$. Then
  \begin{equation*}
    dZ^\e_t = \left(F(X^{\e,m}_{[t]_m}, U^{\e}_{[t]_m}) - F(X^\e_t,U^\e_t)\right) dY^\e_t, \quad Z^\e_0 = 0.
  \end{equation*}
  For $\rho>0$, we define an $\{\F^\e_t\}$-stopping time
  \begin{equation*}
    T^{\e,m,\rho} := \inf\{t\ge0: | X^{\e,m}_t - X^{\e,m}_{[t]_m} | + | U^{\e}_t - U^{\e}_{[t]_m} | \ge \rho \}.
  \end{equation*}
  Then for $t\in[0,T^{\e,m,\rho}]$, we have
  \begin{equation*}
    \begin{split}
       | F(X^{\e,m}_{[t]_m}, U^{\e}_{[t]_m}) - F(X^\e_t,U^\e_t)| & \le | F(X^{\e,m}_{[t]_m}, U^{\e}_{[t]_m}) - F(X^{\e,m}_t,U^\e_t)| + | F(X^{\e,m}_t,U^\e_t) - F(X^\e_t,U^\e_t)| \\
         & \le \|F\|_\Lip(\rho+ |X^{\e,m}_t-X^\e_t|).
    \end{split}
  \end{equation*}
  Together with Lemma \ref{est-8}, it yields that for any $a>0$ and $0<\e \le 1$,
  \begin{equation*}
    \e\log\P^\e \left( \sup_{t\in[0,T^{\e,m,\rho}]} |X^{\e,m}_t - X^\e_t| \ge a\right) \le K'+ \log\left( \frac{\rho^2}{\rho^2+a^2}\right).
  \end{equation*}
  Hence,
  \begin{equation}\label{est-1}
    \lim_{\rho\to 0}\sup_{m\in\N_+}\limsup_{\e\to0}\e\log\P^\e \left( \sup_{t\in[0,T^{\e,m,\rho}]} |X^{\e,m}_t - X^\e_t| \ge a\right) = -\infty.
  \end{equation}
  On the other hand, we note that
  \begin{equation*}
    | X^{\e,m}_t - X^{\e,m}_{[t]_m} | = |F(X^{\e,m}_{[t]_m}, U^{\e}_{[t]_m})(Y^\e_t - Y^\e_{[t]_m})| \le \|F\|_0 w_T(Y^\e,1/m).
  \end{equation*}
  Thus we have
  \begin{equation*}
    \begin{split}
      \P^\e\{T^{\e,m,\rho} < T\} &= \P^\e\left(\sup_{0\le t\le T} \left( | X^{\e,m}_t - X^{\e,m}_{[t]_m} | + | U^{\e}_t - U^{\e}_{[t]_m} | \right) \ge \rho\right) \\
      &\le \P^\e\left(\sup_{0\le t\le T} | X^{\e,m}_t - X^{\e,m}_{[t]_m} | \ge \frac{\rho}{2}\right) + \P^\e\left(\sup_{0\le t\le T} | U^{\e}_t - U^{\e}_{[t]_m} | \ge \frac{\rho}{2}\right) \\
      &\le \P^\e(w_T(Y^\e,1/m) \ge \rho/2) + \P^\e(w_T(U^\e,1/m) \ge \rho/2).
    \end{split}
  \end{equation*}
  By virtue of the exponential tightness of $(Y^\e,U^\e)$ and Lemma \ref{exp-tight}, for each $T>0$ and $\rho>0$
  \begin{equation}\label{est-2}
    \lim_{m\to\infty}\limsup_{\e\to0}\e\log\P^\e\{T^{\e,m,\rho} < T\} = -\infty.
  \end{equation}
  Note that for all $T>0$,
  \begin{equation*}
    \{\|X^{\e,m} - X^\e\|_T >a\} \subset \left\{\sup_{t\in[0,T^{\e,m,\rho}]} |X^{\e,m}_t - X^\e_t| \ge a\right\} \cup \{T^{\e,m,\rho} < T\}.
  \end{equation*}
  Combining this with \eqref{est-1} and \eqref{est-2}, we obtain the desired result.
\end{proof}

\begin{corollary}\label{exp-tight-X}
  Under the assumptions in Lemma \ref{exp-good-app}, the family $\{X^\e\}_{\e>0}$ is exponentially tight.
\end{corollary}
\begin{proof}
  We first show that for each $m\in\N_+$, the family $\{X^{\e,m}\}_{\e>0}$ defined in \eqref{SDE-app} obeys the LDP with a good rate function. To this end, we define for each $m$ a map $\Psi^m: \C^{n+l} \to \C^d$ via $f = \Psi^m(g,h)$, where $f(0) = 0$ and for $t\in(N+\frac{k}{m},N+\frac{k+1}{m}]$, $N\in\N$, $k = 0,1,\cdots,m-1$,
  \begin{equation*}{\textstyle
    f(t) = f(N+\frac{k}{m}) + F(f(N+\frac{k}{m}),h(N+\frac{k}{m})) \left(g(t) - g(N+\frac{k}{m}) \right). }
  \end{equation*}
  Then it is easy to see $X^{\e,m} = \Psi^m(Y^\e,U^\e)$. Let $f_i = \Psi^m(g_i,h_i)$ for $g_i\in\C^n,h_i\in\C^l$, $i=1,2$. Let $e = f_1 - f_2$. Then for fixed $N\in\N$,
  \begin{equation*}
    \sup_{t\in(N+\frac{k}{m},N+\frac{k+1}{m}]} e(t) \le \left(1+2\|F\|_\Lip (\|g_1\|_{N+1}+1)\right) \left( e({\textstyle N+\frac{k}{m}}) + \|g_1-g_2\|_{N+1} + \|h_1-h_2\|_{N+1}\right).
  \end{equation*}
  Since $e(0) = 0$, by iterating this bound over $k = 0,1,\cdots,m-1$ and $N$, we have
  \begin{equation*}
    \|e(t)\|_{N+1} \le C(\|F\|_\Lip,\|g_1\|_{N+1},N) \left(\|g_1-g_2\|_{N+1} + \|h_1-h_2\|_{N+1}\right).
  \end{equation*}
  Hence $\Psi^m$ is continuous from $\C([0,N+1];\R^{n+l})$ to $\C([0,N+1];\R^d)$ for each $N\in\N$, and consequently continuous from $\C^{n+l}$ to $\C^d$ (cf. \cite[Proposition IV.5.6, IV.5.7]{Con90}). The contraction principle (see, e.g., \cite[Theorem 4.2.1]{DZ98}) yields that the family $\{X^{\e,m}\}$ obeys the LDP with a good rate function, and thus is exponentially tight for each $m$.

  Fix $T>0$. Observe that for all $a>0$, $\eta>0$ and $\rho>0$,
  \begin{align*}
    \{\|X^\e\|_T \ge a\} &\subset \{\|X^\e-X^{\e,m}\|_T \ge {\textstyle\frac{a}{2}} \} \cup \{ \|X^{\e,m}\|_T \ge {\textstyle\frac{a}{2}} \}, \\
    \{w_T(X^\e,\rho) \ge \eta \} &\subset \{\|X^\e-X^{\e,m}\|_T \ge {\textstyle\frac{\eta}{4}} \} \cup \{w_T(X^\e,\rho) \ge {\textstyle\frac{\eta}{2}} \}.
  \end{align*}
  Then the exponential tightness of the family $\{X^\e\}$ follows from that of $\{X^{\e,m}\}$, Lemma \ref{exp-good-app} and Lemma \ref{exp-tight}.
\end{proof}

\subsection{Proofs}

\begin{proof}[\textbf{Proof of Proposition \ref{LDP-control}}]
  \emph{Step 1 (Identification of the rate function).} Suppose that the family $\{(X^\e,Y^\e,U^\e)\}_{\e>0}$ is exponentially tight. We will show that for any subsequence $\{(X^{\e_k},Y^{\e_k},U^{\e_k})\}_{k=1}^\infty$, with $\e_k\to0$ as $k\to\infty$, which obeys the LDP, the rate function $I$ is given by \eqref{rate-func}. For notational simplicity, we still denote the subsequence $\e_k$ by $\e$.

  We follow the lines of \cite[Theorem 6.1]{Gan18}. By the contraction principle, the family $(Y^\e,U^\e,F(X^\e,U^\e))$ obeys the LDP with the good rate funtion $I^\sharp(y,u,f) = \inf\{I(x,y,u): f=F(x,u)\}$. Since $X^\e = F(X^\e,U^\e)\cdot Y^\e$, Lemma \ref{LDP-integral} yields that the family $(X^\e,Y^\e,U^\e,F(X^\e,U^\e))$ obeys the LDP with the good rate function
  \begin{equation}\label{rate-J}
    \begin{split}
      J(x,y,u,f) & =
      \begin{cases}
        I^\sharp(y,u,f), & x = f\cdot y \text{ and } y \text{ is locally of finite variation}, \\
        \infty, & \text{otherwise}.
      \end{cases} \\
         & =
      \begin{cases}
        \inf\{I(x',y,u): f=F(x',u)\}, & x = f\cdot y \text{ and } y \text{ is locally of finite variation}, \\
        \infty, & \text{otherwise}.
      \end{cases}
    \end{split}
  \end{equation}
  But the contraction principle yields $I(x,y,u) = \inf_f \{J(x,y,u,f)\}$. Hence, if $y$ is of infinite variation, then by \eqref{rate-J}, $J(x,y,u,f) = \infty$ and $I(x,y,u) = \infty$.

  On the other hand, using contraction principle once again, the rate function $J$ is
  \begin{equation}\label{rate-J-2}
    J(x,y,u,f) =
    \begin{cases}
      I(x,y,u), & f = F(x,u), \\
      \infty, & \text{otherwise}.
    \end{cases}
  \end{equation}
  Suppose $y$ is locally of finite variation but $x\ne F(x,u)\cdot y$, we claim that $J(x,y,u,f) = \infty$ and then $I(x,y,u) = \infty$. If $x \ne f\cdot y$, then \eqref{rate-J} yields the claim. If $x = f\cdot y$, then $f \ne F(x,u)$, and the claim follow from \eqref{rate-J-2}.

  Suppose now $I(x,y,u)<\infty$. Then the previous arguments yield that $y$ is locally of finite variation and $x$ satisfies the equation $x = F(x,u)\cdot y$. Again by the contraction principle, $I'(y,u) = \inf_{x'}\{I(x',y,u)\}$, and obviously $I'(y,u) \le I(x,y,u)$. If $I'(y,u) <I(x,y,u)$, then there exists $x'$ such that $I(x',y,u)<I(x,y,u) <\infty$. Hence, $x'$ satisfies the equation $x' = F(x',u)\cdot y$, which yields $x = x'$, since this equation has a unique solution as $F$ is Lipschitz. Therefore, we have $I'(y,u) =I(x,y,u)$ in this case. The representation \eqref{rate-func} follows.

  In the rest steps, we only need to show the exponential tightness of $\{(X^\e,Y^\e,U^\e)\}_{\e>0}$, provided the exponential tightness of $\{(Y^\e,U^\e)\}_{\e>0}$ and $\{G(Y^\e)_t\}_{\e>0}$ for each $t$.

  \emph{Step 2 (Localization).} Suppose $\{(X^\e,Y^\e, U^\e)\}$ is exponentially tight if in addition, the following conditions are satisfied:
  the family of increasing processes $\{G(Y^\e)\}$ are uniformly bounded. We deduce it holds in general.


  For each $\e, p>0$, define a stopping time
  $$T^{\e,p}:=\inf\{t\ge0: G(Y^\e)_t \ge p\}.$$
  Then $T^{\e,p}$ is nondecreasing in $p$. Let $U^{\e,p} := U^\e_{\cdot\wedge T^{\e,p} }$, $Y^{\e,p} := Y^\e_{\cdot\wedge T^{\e,p} }$, and each $X^{\e,p}$ be the solution of the SDE
  \begin{equation}\label{sde-app}
    dX^{\e,p}_t = F(X^{\e,p}_{t-},U^{\e,p}_{t-}) dY^{\e,p}_t, \quad X^{\e,p}_0 = x_0,
  \end{equation}
  Then $X^{\e,p} = X^\e_{\cdot\wedge T^{\e,p} }$. It is obvious that (cf. \cite[Eq. (VI.1.9)]{JS13})
  \begin{align*}
    \sup_{t\le T}(|Y^{\e,p}_t| + |U^{\e,p}_t|) &= \sup_{t\le T \wedge T^{\e,p}}(|Y^{\e}_t| + |U^{\e}_t|) \le \sup_{t\le T }(|Y^{\e}_t| + |U^{\e}_t|), \\
    w_T((Y^{\e,p},U^{\e,p}),\rho) &= w_{T\wedge T^{\e,p}}((Y^{\e},U^{\e}),\rho) \le w_{T}((Y^{\e},U^{\e}),\rho).
  \end{align*}
  Lemma \ref{exp-tight} yields that $\{(Y^{\e,p},U^{\e,p})\}$ is also exponentially tight, for each $p>0$. Note that the increasing process $G(Y^{\e,p})$ associated to each $Y^{\e,p}$ coincides with $G(Y^{\e})_{\cdot\wedge T^{\e,p} }$. Hence, $G(Y^{\e,p})$ is uniformly bounded by $p$, for each $p>0$, and $G(Y^{\e,p})\le G(Y^{\e})$, which yields the exponential tightness for the family $\{G(Y^{\e,p})_t\}_{\e>0}$, for each $t>0$ and $p>0$. Then our assumption implies that each family $\{(X^{\e,p},Y^{\e,p},U^{\e,p})\}_{\e>0}$ is exponentially tight for every $p>0$. Using Lemma \ref{exp-tight} again, for each $p>0$ and all $T>0$, $\eta>0$, and any $\delta>0$, there exist $a>0$, $\rho>0$ and $\e'_0>0$, such that for all $0<\e\le\e'_0$,
  \begin{align}
    \P^\e\left( \sup_{t\le T}(|U^{\e,p}_t| + |Y^{\e,p}_t| + |X^{\e,p}_t|)\ge a \right) &< \delta^{1/\e}, \label{est-4}\\
    \P^\e\left( w_T((U^{\e,p},Y^{\e,p},X^{\e,p}),\rho) \ge\eta \right) &< \delta^{1/\e}, \label{est-5}
  \end{align}

  Note that for each $\e,p, T>0$,
  \begin{equation}\label{est-9}
    \{T^{\e,p} > T\} \subset \{(U^{\e,p}_t,Y^{\e,p}_t,X^{\e,p}_t) = (U^\e_t,Y^\e_t,X^\e_t), \forall t \in[0,T] \}.
  \end{equation}
  Fix $T>0$. Using the exponential tightness of $\{G(Y^\e)_T\}$, there exist $p_0>0$ and $\e_0>0$, such that for all $0<\e\le\e_0$,
  \begin{equation}\label{est-6}
    \P^\e(T^{\e,p_0} \le T) = \P^\e\left( G(Y^\e)_T \ge p_0 \right) < \delta^{1/\e}.
  \end{equation}
  Hence, combining \eqref{est-4}, \eqref{est-5}, \eqref{est-9} and \eqref{est-6}, for all $0<\e\le \e_0 \wedge \e'_0 \wedge 1$,
  \begin{equation*}
    \begin{split}
       &\ \P^\e\left( \sup_{t\le T}(|U^{\e}_t| + |Y^{\e}_t| + |X^{\e}_t|)\ge a \right) \\
       =&\ \P^\e\left( \sup_{t\le T}(|U^{\e}_t| + |Y^{\e}_t| + |X^{\e}_t|)\ge a, T^{\e,p_0} > T \right) + \P^\e\left( \sup_{t\le T}(|U^{\e}_t| + |Y^{\e}_t| + |X^{\e}_t|)\ge a, T^{\e,p_0} \le T \right) \\
       \le&\ \P^\e\left( \sup_{t\le T}(|U^{\e,p_0}_t| + |Y^{\e,p_0}_t| + |X^{\e,p_0}_t|)\ge a \right) + \P^\e\left( T^{\e,p_0} \le T \right)  \\
       < &\ 2\delta^{1/\e},
    \end{split}
  \end{equation*}
  and
  \begin{equation*}
    \P^\e\left( w_T((U^{\e},Y^{\e},X^{\e}),\rho) \ge\eta \right) \le \P^\e\left( w_T((U^{\e,p_0},Y^{\e,p_0},X^{\e,p_0}),\rho) \ge\eta \right) + \P^\e\left( T^{\e,p_0} \le T \right) < 2\delta^{1/\e}.
  \end{equation*}
  Once again, Lemma \ref{exp-tight} says that the family $\{(X^\e,Y^{\e},U^{\e})\}$ is exponentially tight.

  \emph{Step 3 (Exponential tightness of $\{(X^\e,Y^\e,U^\e)\}$).} In this step, we will assume the family of processes $\{G(Y^\e)\}$ to be uniformly bounded by a constant $K>0$. By Corollary \ref{exp-tight-X}, the family $\{X^\e\}$ is exponentially tight. Observe that for all $a>0$, $\eta>0$ and $\rho>0$,
  \begin{align*}
    \{\|(X^\e,Y^\e,U^\e)\|_T \ge a\} &\subset \{\|X^\e\|_T \ge {\textstyle\frac{a}{2}} \} \cup \{ \|(Y^\e,U^\e)\|_T \ge {\textstyle\frac{a}{2}} \}, \\
    \{w_T((X^\e,Y^\e,U^\e),\rho) \ge \eta \} &\subset \{w_T(X^\e,\rho) \ge {\textstyle\frac{\eta}{2}} \} \cup \{w_T((Y^\e,U^\e),\rho) \ge {\textstyle\frac{\eta}{2}} \}.
  \end{align*}
  The exponential tightness of $\{(X^\e,Y^\e,U^\e)\}$ follows from that of $\{X^\e\}$ and $\{(Y^\e,U^\e)\}$, together with Lemma \ref{exp-tight}.
\end{proof}

\begin{proof}[\textbf{Proof of Corollary \ref{local-Lip}}]
  Since we have already identify the rate function in Step 1 of the proof of Proposition \ref{LDP-control}, it is only needed to show the family $\{(X^\e,Y^\e,U^\e)\}$ is exponentially tight. Suppose first that the families $\{X^\e\}$ and $\{U^\e\}$ are uniformly bounded by a constant $K>0$. Let $f:\R^d\times\R^l \to \R^d\times\R^l$ be a 
  bounded and continuously differentiable function with $f(x,u) = (x,u)$ when $|x|\le K$ and $|u|\le K$.
  Define $\tilde F(x,u) := F(f(x,u))$. Then obviously $\tilde F$ is global Lipschitz and each $X^\e$ is the solution of the following SDE:
  \begin{equation*}
    dX^\e_t = \tilde F(X^\e_{t},U^\e_{t}) dY^\e_t, \quad X^\e_0 = x_0.
  \end{equation*}
  Hence, Proposition \ref{LDP-control} yields that $\{(X^\e,Y^\e,U^\e)\}$ is exponential tight. Now by virtue of the assumption and the exponential tightness of $\{U^\e\}$, the family $\{\sup_{0\le s\le t}(|X^\e_s|+|U^\e_s|)\}$ is exponential tight for every $t>0$. The general case follows from a similar localization argument in Step 2 of the proof of Proposition \ref{LDP-control}, with $\sup_{0\le s\le t}(|X^\e_s|+|U^\e_s|)$ in place of $G(Y^\e)_t$.
\end{proof}

\paragraph{Acknowledgements.}
The research of J. Duan was partly supported by the NSF grant 1620449. The research of Q. Huang was partly supported by NSFC grants 11531006 and 11771449. We would like to thank Prof.~Jifu Xiao for helpful discussions.

\end{document}